\documentclass[10pt]{amsart}
\usepackage[top=1.3in, bottom=1.25in, left=1.25in, right=1.25in]{geometry}

\usepackage{amssymb}
\usepackage{mathrsfs}
\usepackage{mathtools}
\usepackage{graphicx}
\usepackage{scrextend}
\usepackage{eucal}
\usepackage{amsmath}
\usepackage{amsthm}
\usepackage{amsfonts}
\usepackage{bm}
\usepackage[all,cmtip]{xy}
\usepackage{hyperref}
\usepackage{color}

\newtheorem{thm}{Theorem}[section]
\newtheorem{prop}[thm]{Proposition}

\newtheorem{cor}[thm]{Corollary}

\theoremstyle{definition}
\newtheorem{definition}[thm]{Definition}

\newtheorem{lemma}[thm]{Lemma}

\theoremstyle{remark}
\newtheorem{remark}[thm]{Remark}

\numberwithin{equation}{section}

\newcommand{\R}{\mathbf{R}}  

\DeclareMathOperator{\Def}{Def}
\DeclareMathOperator{\Defab}{Def_{(a,b)}}

\DeclareMathOperator{\Res}{Res}
\DeclareMathOperator{\End}{End}
\DeclareMathOperator{\Aut}{Aut}

\DeclareMathOperator{\can}{can}
\DeclareMathOperator{\uu}{u}
\DeclareMathOperator{\ppdiv}{pdiv}
\DeclareMathOperator{\GL}{GL}
\DeclareMathOperator{\GSp}{GSp}

\DeclareMathOperator{\tor}{tor}
\DeclareMathOperator{\an}{an}
\DeclareMathOperator{\Gal}{Gal}
\DeclareMathOperator{\Spf}{Spf}
\DeclareMathOperator{\Spec}{Spec}

\newcommand{\Rab}{\ensuremath{R_{(a,b)}}}

\newcommand{\Uab}{\ensuremath{U_{(a,b)}}}

\newcommand{\pdiva}{\ensuremath{\mathscr{G}_{(a,r)}}}
\newcommand{\pdivb}{\ensuremath{\mathscr{G}_{(b,r)}}}
\newcommand{\pdivab}{\ensuremath{\mathscr{G}_{(a,b)}}}
\newcommand{\C}{\ensuremath{\mathbb{C}}}

\newcommand{\Z}{\ensuremath{\mathbb{Z}}}

\newcommand{\D}{\ensuremath{\mathbb{D}}}

\newcommand{\A}{\ensuremath{\mathbb{A}}}

\newcommand{\F}{\ensuremath{\mathbb{F}}}
\newcommand{\Fpbar}{\ensuremath{\overline{\mathbb{F}}_p}}

\newcommand{\Q}{\ensuremath{\mathbb{Q}}}

\newcommand{\Ok}{\ensuremath{\mathcal{O}}}

\newcommand{\pdivh}{\ensuremath{\mathscr{H}}}
\newcommand{\pdiv}{\ensuremath{\mathscr{G}}}
\newcommand{\lpdiv}{\ensuremath{\tilde{\mathscr{G}}}}

\newcommand{\et}{\ensuremath{{\acute{e}t}}}

\newcommand{\Hom}{\ensuremath{\mbox{Hom}}}

\newcommand{\Ext}{\ensuremath{\mathscr{Ext}}}

\begin{document}
	\title{Serre--Tate theory for Shimura varieties of Hodge type} 
	\author{Ananth N. Shankar}
	\address{Department of Mathematics, Massachusetts Institute of Technology, Cambridge, USA, MA 02139}
		\email{ananths@mit.edu}
	\author{Rong Zhou}
	\address{Department of Pure Mathematics and Mathematical Statistics, University of Cambridge,  Cambridge, UK, CB3 0WA}
	\email{rz240@cam.ac.uk}	
	
	\maketitle
	\begin{abstract}
		We study the formal neighbourhood of a point in the $\mu$-ordinary locus of an integral model of a Hodge type Shimura variety. We show that this formal neighbourhood has a  structure of a ``shifted cascade''. Moreover we show that the CM points on the formal neighbourhood are dense and that the identity section of the shifted cascade corresponds to a lift of the abelian variety which has a characterization in terms of its endomorphisms, analogous to the Serre--Tate canonical lift of an ordinary abelian variety.
	\end{abstract}
	

	\tableofcontents
	
	\section{Introduction}
	
	Given an abelian variety $\mathcal{A}$ over a perfect field $k$ of characteristic $p$, the classical theorem of Serre and Tate (\cite{ST}) says that deformations of $\mathcal{A}$ correspond to deformations of the associated $p$-divisible group $\mathcal{A}[p^\infty].$  Using this they discovered a remarkable and surprising group structure on the deformation space of a principally polarized ordinary abelian variety $\mathcal{A}$. More precisely, they  showed that this deformation space has the structure of a formal torus over the Witt vectors of $k$. Moreover, if $\mathcal{A}$ has CM, then the torsion points of this formal torus correspond precisely to deformations of $\mathcal{A}$ which also have CM. The identity section of the torus corresponds to what is known as the Serre--Tate canonical lift. This can  be characterized as the unique deformation of $\mathcal{A}$ to $W(k)$ to which all automorphisms of $\mathcal{A}$ extend.  We refer to \cite{Katz} for more details about these results.
	
	The deformation space of $\mathcal{A}$ can be identified with the formal neighbourhood of the corresponding point on the ordinary locus of $\mathscr{A}_g$, the moduli space of principally polarized abelian varieties. Therefore, this result gives some very interesting  and useful information about the local structure of $\mathscr{A}_g$. For example, this description of the formal neighbourhood was a key input in Chai's verification of certain cases of the Hecke orbit conjecture; see \cite{Chai1}.

	In this paper we propose a generalization of the above results of Serre and Tate to more general Shimura varieties, specifically those of Hodge type. To explain the results we need to introduce some notation.

	Let $(G,X)$ be a Shimura datum and $p>2$ a prime, we assume it is equipped with a Hodge embedding $\rho:G\rightarrow \GSp(V,\psi)$ in the sense of \cite{Del}. Assume the base change $G_{\Q_p}$ has a reductive model $G_{\Z_p}$ over $\Z_p$ (this is equivalent to $G_{\Q_p}$ being unramified). Then for $K_p=G_{\Z_p}(\Z_p)$ and $K^p\subset G(\A_f^p)$ sufficiently small compact open subgroups (here $\A_f^p$ is the ring of finite adeles with trivial component at $p$), Kisin \cite{Ki2} has constructed the canonical integral model $\mathscr{S}_{K_pK^p}(G,X)/\mathcal{O}_{E_v}$ for the Shimura variety $Sh_{K_pK^p}(G,X)/E$.  Here $E$ is the reflex field, $v$ a prime of $E$ above $p$ and $\mathcal{O}_{E,v}$ the ring of integers of the completion $E_v$ of $E$.
	
	Given a point $x\in\mathscr{S}_{K_pK^p}(G,X)(\Fpbar)$, it is shown in \cite{Ki2} that one can associate to $x$ an abelian variety $\mathcal{A}_x$ with $G$-structure (see \S 2) and its $p$-divisible group $\pdiv_x$. We write $\Ok_L$ for $W(\Fpbar)$. The main result of the paper is the following
	
	\begin{thm} \label{main}
		
		Let $x\in \mathscr{S}_{K_pK^p}(G,X)(\Fpbar)$ lie on the $\mu$-ordinary locus and let $\widehat{U}_x$ denote the formal neighbourhood of $\mathscr{S}_{K_pK^p}(G,X)$ at $x$. Then, 
		\begin{enumerate}
			
			\item There exists exists a unique point $\tilde{x}\in \widehat{U}_x(\Ok_L)$ which can be characterized as the unique $\Ok_L$-point lifting $x$ for which all $G$-automorphisms of $\mathcal{A}_x$ lift to $\mathcal{A}_{\tilde{x}}$. \label{mainone}
			\item $\widehat{U}_x$ has the structure of a shifted subcascade.\label{maintwo}
			
			\item The set of CM points in $\widehat{U}_x$ are dense. \label{mainthree}
		\end{enumerate}
	\end{thm}
	
	The $\mu$-ordinary locus is the correct analogue of the ordinary locus of $\mathscr{A}_g$ in our setting. This notion was introduced by Rapoport (during the 1996 Langlands conference). See also Wortmann \cite{Wo}. We refer to \S 6 for the precise definitions. The special lift constructed in Theorem \ref{main} is the analogue of the Serre--Tate canonical lift and the notion of shifted cascade is the correct generalization of the formal group structure to our context.
	
	\subsection{Past results and related work}
	These results have been obtained by Moonen \cite{Mo} in the case of PEL type Shimura varieties. However, the Shimura varieties we consider do not have a moduli interpretation unlike in the PEL case. Therefore our techniques are necessarily very different, and our proofs are often simpler than in the PEL case.
	
	In the special case when $\pdiv_x$ has two slopes, the shifted subcascade is actually a $p$-divisible formal group. When the two slopes are zero and one (that is, when the special fiber intersects the ordinary locus of $\mathscr{A}_g$), Theorem \ref{main} is known due to work of  Noot (\cite{Noot}). His proof crucially uses work on ordinary F-crystals due to Deligne and Illusie (\cite{DI}). 
	
	In a similar direction, Chai in \cite{Chai} considers the equicharacteristic-$p$ deformation space of a $p$-divisible group in the two-slope case. He proves that the formal neighbourhood of the Newton stratum inside the deformation space has the structure of a formal group, and that the completion of the central leaf (see {\it loc. cit.} for the definitions of central and isogeny leaves) corresponds to the $p$-divisible formal part of this formal group. Chai and Oort also study formal neighbourhoods of central leaves in greater generality than in just the two-slope case. Hamacher in \cite{Ham} obtains Serre-Tate coordinates for central leaves in the unramified Hodge-type case. We note that these are purely equi-characteristic $p$ results.
	
		Some of the results in our paper have also been obtained by Hong in \cite{Hong} for the two-slope case using a different method.
	\subsection{Outline of the proof}
	
	Our proofs rely on the following three inputs:
	\begin{enumerate}
		
		\item The introduction of group-theoretic techniques to prove certain results on crystals with $G$-structure. This simplifies and generalizes some results in Chapter 1 of \cite{Mo}.
		
		\item The smoothness of the integral models which allows us to check certain properties on the level of $\Ok_L$-points.
		
		\item Fontaine--Laffaile theory, which allows us to compute with $\Ok_L$-points of the Shimura variety. 
	\end{enumerate}
	We now sketch in detail the proof of Theorem \ref{main} in the two-slope case. Suppose that the isoclinic parts of $\pdiv_x$ are $\pdiv_1$ and $\pdiv_2$. We first prove that the slope filtration of the $p$-divisible group extends to the family of $p$-divisible groups over $\widehat{U}_x$. The precise result is stated as Proposition \ref{slopefilt};  the proof of this uses the explicit description of the deformation space in \cite{Ki1}. It follows that $\widehat{U}_x$ is a formal subscheme of an extension group. To prove that $\widehat{U}_x$ is a formal subgroup, we use its smoothness to reduce the problem to the case of $\Ok_L$-points. By Grothendieck-Messing theory, every $\Ok_L$-point of $\widehat{U}_x$ corresponds to a filtration of the Dieudonn\'e module of $\pdiv_x$, and therefore a Honda system.
	
	Using work of Kisin (\cite{Ki2}), it is possible to exactly characterize the filtrations which correspond to $\Ok_L$-points of $\widehat{U}_x$. Indeed, the filtrations which arise correspond to a $\Ok_L$-points of a very natural unipotent subgroup $U_G \subset G$ associated to the point $x$. To finish the proof, we use Fontaine--Laffaille theory to compute the Baer sum of two such Honda systems, and prove that this is compatible with the group structure on $U_G(\Ok_L)$. It then follows that the resulting sum lies in $\widehat{U}_x$, as required.

	\subsection{Organization of the paper}
	We now give an outline of the paper. In sections 2 and 3, we recall some group theoretic results and use this to study the $G_{\Ok_L}$-adapted deformations of a $\mu$-ordinary $p$-divisible group with $G$-structure.  The main result in this section of a canonical lifting of the $p$-divisible which is characterized in terms of its automorphisms. For this we use a result of Wortmann \cite{Wo} on $F$-crystals with $G$-structure. Using this, the lift can then be constructed using Grothendieck-Messing theory. We also prove that there is a slope filtration on $p$-divisible groups lying in the $\mu$-ordinary locus. 
	
	In section 4, we recall the definition of the subspace $\Spf R_G$ of the universal deformation space of a $p$-divisible group with $G$-structure which cuts out the $G_{\Ok_L}$-adapted liftings. In the case when the $p$-divisible group is $\mu$-ordinary, we prove the existence of a canonical slope filtration on the versal $p$-divisible group over $\Spf R_G$. In \S 5, we compute with Fontaine-Laffaile modules to prove that $\Spf R_G$ naturally has the structure of a shifted subspace of a cascade. In \S 6, we apply the previous results to study the formal neighbourhood $\widehat{U}_x$ at a $\mu$-ordinary point of a Hodge-type Shimura variety and prove Theorem \ref{main}. 
	
	Some of the results in our paper have been obtained by Hong in \cite{Hong} for the two-slope case using a different method.
	
	{\bf Acknowledgements:} It is a pleasure to thank Tom Lovering for asking a question which led to this paper. We also thank George Boxer and Mark Kisin for useful discussions involving the material in this paper. We are also very grateful to the referee whose comments helped improve the exposition of this work.

	\section{$F$-cystals with $G$-structure}
	In this section we define $F$-crystals with $G$ structure and recall some group theoretic results of \cite{Wo}. These objects naturally arise as the Dieudonn\'e module of a $p$-divisible at a mod $p$ point of a Hodge type Shimura variety.
	
	Let $\Ok_L=W(\Fpbar)$ and $L=\mbox{Frac}(\Ok_L)$. Let $\Gamma=\Gal(\overline{\Q}_p/\Q_p)$ and $\Gamma_L=\Gal(\overline{L}/L)$, then $\Gamma_L$ can be identified with the inertia subgroup of $\Gamma$. We write $\sigma$ for the Frobenius automorphism of $L$.
	
	For any module $M$ over a ring $R$, we write $M^\otimes$ for the direct sum of all tensor products, duals, symmetric and exterior powers of $M$. Let $V$ be a finite free $\Z_p$-module and let $s_{\alpha}\in V^\otimes$ be a collection of tensors whose stabilizer is a reductive group $G$ over $\Z_p$.
	
	\begin{definition}
		An $F$-crystal with $G$-structure on $V$ is the data of an injective  $\sigma$-linear map $$\varphi:V\otimes_{\Z_p}\Ok_L\rightarrow V\otimes_{\Z_p}\Ok_L[\frac{1}{p}]$$
		such that $\varphi(s_\alpha)=s_{\alpha}$.
	\end{definition}
	Given $\varphi_1, \varphi_2$, $F$-crystals with $G$-structure on $V$, an isomorphism between $\varphi_1$ and $\varphi_2$ is an  $\Ok_L$-linear isomorphism $\theta:V\otimes_{\Z_p}\Ok_L\cong V\otimes_{\Z_p}\Ok_L$ such that $\theta(s_\alpha)=s_\alpha$ and we have a commutative diagram
	
\[\xymatrix{	V\otimes_{\Z_p}\Ok_L \ar[r]^{\varphi_1} 	\ar[d]_{\theta} & V\otimes_{\Z_p}\Ok_L[\frac{1}{p}]\ar[d]^{\theta}\\
		V\otimes_{\Z_p}\Ok_L \ar[r]^{\varphi_2} & V\otimes_{\Z_p}\Ok_L[\frac{1}{p}]
	}\]

	When the set of $s_{\alpha}$ is empty, $G=GL(V)$ and we recover the definition of an $F$-crystal, where negative slopes are allowed. 
	
	Given $\varphi$ an $F$-crystal with $G$-structure on $V$, since $\varphi(s_\alpha)=s_\alpha$, we may write $\varphi=b\sigma$ for some $b\in G(L)$. We define an equivalence relation on $G(L)$ by letting $g_1\sim g_2$ if there exists $h\in G(\Ok_L)$ such that $g_1=h^{-1}g_2\sigma(h)$, and we write $G(L)/G(\Ok_L)_\sigma$ to be the set of equivalence classes.
	
	The following proposition is immediate from the definitions.
	
	\begin{prop}
		Let $G$ be as above. 
		\begin{enumerate}
			\item The association $\varphi\mapsto b\in G(L)$ as above defines a bijection between $F$-crystals with $G$-structure on $V$ and the set $G(L)/G(\Ok_L)_\sigma$.
			
			\item The forgetful functor from $F$-crystals with $G$ structure on $V$ to $F$-crystals on $V$ is faithful.
		\end{enumerate}
	\end{prop}
	
	The first part of the proposition tells us that in  order to study $F$-crystals with $G$ structure, it suffices to study the $\sigma$-conjugation action of $G(\Ok_L)$ on $G(L)$.
	
	\subsection{} Let $G$ be a connected reductive group over $\Z_p$, in particular $G_{\Q_p}$ is unramified. Fix a maximal torus $T$ of $G$ which splits over an unramified extension of $\Z_p$ and $B\supset T$ be a Borel subgroup. 
	
	For $b\in G(L)$, let $\nu_b$ denote the Newton cocharacter of $b$, and let $\overline{\nu}_b \in (X_*(T)\otimes_\Z\Q)^+$ denote the dominant representative of the conjugacy class of $\nu_b$. Let $$\kappa_G:G(L)\rightarrow \pi_1(G)_\Gamma$$ denote the Kottwitz homomorphism.
	By \cite{Ko1}, the map $$b\mapsto (\overline{\nu}_b,\kappa_G(b))\in (X_*(T)\otimes_\Z\Q)^+\times \pi_1(G)_\Gamma$$
	factors theough $B(G)$, the set of $\sigma$-conjugacy classes in $G(L)$, and the induced map \begin{equation}\label{Kottwitz} B(G)\rightarrow (X_*(T)\otimes_\Z\Q)^+\times \pi_1(G)_\Gamma\end{equation} is injective.
	
	Let $E\subset L$ be a finite unramified Galois extension of $\Q_p$ over which $T$ splits and for any dominant $\mu\in X_*(T)$ set $$\overline{\mu}=\frac{1}{[E:\Q_p]}\sum_{\tau\in \text{Gal}(E/\Q_p)}\tau(\mu)$$
	We denote by $[\mu]$ the  image of $\mu$ in $\pi_1(G)_\Gamma$.
	
	Define the set $$X^G_\mu(b)=\{g\in G(L)/G(\Ok_L):g^{-1}b\sigma(g)=G(\Ok_L)\mu(p)G(\Ok_L)\}$$
	Similarly, if $M$ is a standard Levi of $G$ defined over $\Z_p$, we define
	$$X^M_\mu(b)=\{g\in M(L)/M(\Ok_L):g^{-1}b\sigma(g)=M(\Ok_L)\mu(p)M(\Ok_L)\}$$
	
	A result of Gashi \cite{Ga} implies that $X^G_\mu(b)$ is non-empty if and only if $$\kappa_G(b)=[\mu], \ \ \ \overline{\nu}_b\leq \overline{\mu}$$
	where for $\mu,\mu'\in (X_*(T)\otimes_\Z\Q)^+$ two dominant cocharacters, we write $\mu\leq\mu'$ if $\mu'-\mu$ is a positive rational linear combination of positive coroots, and $[\mu]$ denotes the image of $\mu$ in $\pi_1(G)_\Gamma$.

	Now suppose $b\in G(L)$ is such that $\nu_b\in X_*(T)$ is dominant, hence defined over $\Z_p$. Let $M_b$ denote the centralizer of $\nu_b=\overline{\nu}_b$ in $G$, then $b\in M_b(L)$ and in particular $b$ is a basic element in $M_b$. Under these assumptions, we have the following which is Proposition 2.5.4. of \cite{CKV}.
	
	\begin{prop}Let $M\supset M_b$ be a standard Levi defined over $\Z_p$. Assume that $\kappa_M(b)=[\mu]$ in $\pi_1(M)_\Gamma$, then the natural map $X_{\mu}^M(b)\rightarrow X^G_\mu(b)$ is a bijection.
	\end{prop}
	
	Let $b\in G(L)$ be arbitrary. By the Cartan decomposition, there exists a dominant cocharacter $\mu\in X_*(T)$ such that $b\in G(\Ok_L)\mu(p)G(\Ok_L)$. Then we have $1\in X^G_{\mu}(b)$ so that $\overline{\nu}_b\leq\overline{\mu}$, by the above. 
	
	\begin{definition}\label{ordinary}
		We say $b$ is $\mu$-ordinary if $b\in G(\Ok_L)\sigma(\mu(p))G(\Ok_L)$ and $\overline{\nu}_b=\overline{\mu}$. Similarly we say an $F$-crystal with $G$-structure on $V$ is $\mu$-ordinary if the associated $b$ in Proposition 2.2 is $\mu$-ordinary.
	\end{definition}
	Note that the notion of $\mu$-ordinariness for $b\in G(L$) only depends on the image of $b$ in $G(L)/G(\Ok_L)_\sigma$.
	
	\subsection{}
	The following is proved in \cite[Proposition 7.2]{Wo}. As we will be making extensive use of it, we recall the proof for the convenience of the reader. We write $\upsilon$ for  the cocharacter $\sigma(\mu)$.
	\begin{prop}\label{cats}
		Let $b\in G(L)$ be $\mu$-ordinary then there exists $g\in G(\Ok_L)$ such that $g^{-1}b\sigma(g)=\upsilon(p)$. 
	\end{prop}
	
	\begin{remark}
		\begin{enumerate}
			\item By injectivity of the map \ref{Kottwitz}, it follows that under the assumptions of the proposition, there exists $g\in G(L)$ such that $g^{-1}b\sigma(g)=\upsilon(p)$ (this is because the Newton vector of $\upsilon(p)$ is just $\overline{\mu}$ and $\kappa_G(b)=[\upsilon]=\kappa_G(\upsilon(p))$).

		\end{enumerate}
	\end{remark}

	The following Lemma will be needed in the proof of Proposition \ref{cats}
	\begin{lemma}\label{central}Let $\mu\in X_*(T)$ be a dominant cocharacter such that $\overline{\mu}$ is central in $G$. Then $\mu$ is central in $G$.
	\end{lemma}
	\begin{proof}A cocharacter $\mu\in X_*(T)$ is central if and only if $\langle\alpha,\mu\rangle=0$ for all positive roots $\alpha$.
		
		Now note that $\sigma^i(\mu)$ is dominant for all $i$ since $\mu$ is dominant and $\sigma$ preserves the dominant chamber. Therefore $$\langle\alpha,\sigma^i(\mu)\rangle\geq 0$$ for all positive roots $\alpha$. Then since $$0=\langle\alpha,\overline{\mu}\rangle=\langle\alpha,\frac{1}{n}(\mu+\sigma(\mu)+...+\sigma^{n-1}(\mu))\rangle\geq 0$$
		for all positive roots $\alpha$, we have  $\langle\alpha,\mu\rangle=0$ for all positive roots $\alpha$, i.e. $\mu$ is central.
	\end{proof}

	\begin{proof}[Proof of Proposition \ref{cats}]
		Let $\delta=\upsilon(p)$. Then $$\overline{\nu}_\delta=\nu_\delta=\overline{\mu}=\overline{\nu}_b$$ where the second equality follows from functoriality and the explicit description of the Newton vector for a torus in \cite{Ko1}, the first equality follows since $\nu_\delta$ is dominant and the last is by definition of $\mu$-ordinariness. We also have $\kappa_G(\delta)=\kappa_G(b)$. Thus $b$ and $\delta$ are $\sigma$-conjugate in $G(L)$. Let $g\in G(L)$ such that $b=g^{-1}\delta\sigma(g)$, then $g\in X_{\upsilon}(\delta)$.
		
		Let $M$ denote the centralizer of $\nu_\delta$, then we have $\delta\in M(L)$. Since $[\upsilon]=\kappa_M(\delta)$, we have by Proposition 2.5.4 of \cite{CKV}, that the natural map $X_\upsilon^M(\delta)\rightarrow X^G_\upsilon(\delta)$, is a bijection. Thus $g=mk$ for some $m\in X^M_\upsilon(\delta)$ and $k\in G(\Ok_L)$, so that 
		$$b=k^{-1}m^{-1}\delta\sigma(m)\sigma(k).$$ 
		
		By $\sigma$-conjugating $b$ by $k$, we may assume that $$b=m^{-1}\delta\sigma(m)=m_1\upsilon(p)m_2\in M(\Ok_L)\upsilon(p)M(\Ok_L)$$
		
		Since $\upsilon$ is $G$-dominant, hence $M$-dominant, and $\overline{\upsilon}=\overline{\mu}$ is central in $M$, by Lemma \ref{central} we have $\upsilon$ is central in $M$. Therefore  $$b=m_1\upsilon(p)m_2=\upsilon(p)h$$ for some $h\in M(\Ok_L)$. By Lang's theorem for $M$, there exists $m'\in M(\Ok_L)$ such that $m'\sigma(m')^{-1}=h$. Therefore $$m'^{-1}b\sigma(m')=m'^{-1}\upsilon (p)h\sigma(m')=\upsilon(p)m'^{-1}h\sigma(m')=\upsilon(p)$$
		
		Since $M(\Ok_L)\subset G(\Ok_L)$, we are done.
		
	\end{proof}
	\subsection{}Let $V$ and $s_\alpha\in V^\otimes$ be as in \S 2.1 such that $G$ is the stabilizer of $s_{\alpha}$. Then the previous proposition can be regarded as a statement about $\mu$-ordinary $F$-crystals with $G$ structure. Indeed we have the following corollary which follows from Proposition 2.2 and Proposition 2.5
	
	\begin{cor}\label{cortwo}
		Let $\varphi':V\otimes_{\Z_p}\mathcal{O}_L\rightarrow V\otimes_{\Z_p}\Ok_L[\frac{1}{p}]$ be a $\mu$-ordinary $F$-crystal with $G$ structure on $V$. Then $\varphi'$ is isomorphc the $F$-crystal with $G$-structure on $V$ given by $\varphi=\upsilon(p)\sigma$. In particular the underlying $F$-crystals are isomorphic.
	\end{cor}
	
	\subsection{} Given a $\mu$-ordinary $F$-crystal with $G$-structure on $V$, we may pick a representative in its isomorphism class such that $\varphi=\upsilon(p)\sigma$. The Newton cocharacter $\nu=\overline{\nu}$ gives a decomposition of $ V\otimes_{\Z_p}\Ok_L= \bigoplus_{i\in \Q} L_i$, where $L_i$  is part of $V\otimes_{\Z_p}\Ok_L$ where $\nu$ acts by weight $i$ on $L_i$. Since $\nu$ is fixed by $\upsilon(p)\sigma$, this induces a decomposition of the underlying $F$-crystal. We call this the slope decomposition.
	
	\section{Applications to $p$-divisible groups}
	In this section we study the $G_{\Ok_L}$-adapted deformations of a $p$-divisible group equipped with tensors in its Dieudonn\'e module. When $\pdiv$ is $\mu$-ordinary we show the existence of a canonical lift which has a characterization in terms of its automorphisms. We being by recalling the notion of $p$-divisible groups with $G$ structure and $G_{\Ok_L}$-adapted deformations following \cite{Ki2}. As before $G$ is a connected reductive group over $\Z_p$, with $B$ a Borel subgroup of $G$ containg a maximal torus $T$.
	
	\subsection{}Given a $p$-divisible $\pdiv$ over a ring $R$, we write $\D(\pdiv)$ for the contavariant Dieudonn\'e crystal associated to $\pdiv$ as in \cite{Me}. When $\pdiv$ is defined over $\Fpbar$, we also write $\D(\pdiv)=\D(\pdiv)(\Ok_L)$ for its contravariant Diedonn\'e module when there is no risk of confusion.

	\begin{definition}
		A $p$-divisible group with $G$-structure is the data of a $p$-divisible group  $\pdiv$ over $\overline{\F}_p$ and Frobenius-stable tensors $s_{\alpha,0} \in \D(\pdiv)^{\otimes}$, such that:
		\begin{enumerate}
			\item The exists a finite free $\Z_p$-module $V$ together with an $\Ok_L$- linear isomorphism $$V\otimes_{\Z_p}\Ok_L\cong \D(\pdiv)$$ such that $s_{\alpha,0}\in V^\otimes$. Moreover $G\subset GL(V)$ is the stabilizer of the tensors $s_{\alpha,0}$.
			
			\item The stabilizer of the $s_{\alpha,0}$ in $GL(V)$ is $G$.
		\end{enumerate}
	\end{definition}
	Let $G_{\Ok_L}\subset GL(\D(\pdiv))$ denote the stabilizer of $s_{\alpha}$, then an isomorphism $V\otimes_{\Z_p}\Ok_L\cong \D(\pdiv)$ gives an identification of $G_{\Ok_L}\cong G\otimes_{\Z_p}\Ok_L$. Changing the $V\otimes_{\Z_p}\Ok_L\cong \D(\pdiv)$, changes the identification $V\otimes_{\Z_p}\Ok_L\cong \D(\pdiv)$ by conjugation by an element of $G(\Ok_L)$. 
	
	Let $\pdiv$ be a $p$-divisible group with $G$-structure. The Dieudonn\'e module $\D(\pdiv)$ of a $p$-divisible group with $G$ structure is an $F$-crystal with $G$-structure on $V$. Note that picking a different isomorphism $V\otimes_{\Z_p}\Ok_L\cong \D(\pdiv)$ will give an isomorphic $F$-crystal with $G$-struture.  We say that $\pdiv$ is $\mu$-ordinary if the associated $F$-crystal with $G$-structure is $\mu$-ordinary. 
	
	\subsection{}Let $\pdiv$ be a $p$-divisible group with $G$-structure. We now recall the notion of $G_{\Ok_L}$-adapted liftings of $\pdiv$. Let $K/L$ be a finite extension. Let $E(u) \in \Ok_L[u]$ be the Eisenstein polynomial of a fixed uniformiser $\pi$ in $K$. Define the ring $S$ to be the $p$-adic completion of $\Ok_L[u, E(u)/n!]_{n \geq 1}$. Then for $p>2$ the map $S\rightarrow \Ok_K$ is equipped with topologically nilpotent divided powers. The following definition appears in \cite[\S 1.1.8]{Ki2}:
	
	\begin{definition}\label{Gadapted}
		Let $\tilde{\pdiv}$ be a deformation of $\pdiv$ to $\Ok_K$. We say that $\tilde{\pdiv}$ is $G_{\Ok_L}$-adapted if there exist Frobenius invariant tensors $\tilde{s}_{\alpha} \in \D(\tilde{\pdiv})(S)^{\otimes}$ lifting $s_{\alpha,0}$, which define a reductive subgroup $G_S$ and such that the images of the $\tilde{s}_{\alpha}$ in $\D(\tilde{\pdiv})(\Ok_K)^{\otimes}$ are in $\mathrm{Fil}^0(\D(\tilde{\pdiv})(\Ok_K)^{\otimes})$.
	\end{definition} 
	
	We have the following which is \cite[Proposition 1.1.10]{Ki2}
	
	\begin{prop}
		Let $\Spf R$ be the universal deformation space of the underlying $p$-divisible group $\pdiv$. There exists a formally smooth quotient $R_G$ of $R$ such that for any $\varpi:R\rightarrow K$, the induced $p$-divisible group $\pdiv_{\varpi}$ is $G_{\Ok_L}$-adapted if and only if $\varpi$ factors through $R_G$.
	\end{prop}
	
	We will recall the construction of $R$ and $R_G$ in the next section. For now we give a more explicit description of the $G_{\Ok_L}$-adapted liftings to $\Ok_L$. 
	
	By Grothendieck-Messing theory, lifts of $\pdiv$ to $\Ok_L$ are characterized by lifts of the natural filtration on $\D(\pdiv)\otimes_{\Z_p}({\F}_p)$ to $\D(\pdiv)$.  We fix an isomorphism $V\otimes_{\Z_p}\Ok_L$, and hence an identification $G\otimes_{\Z_p}\Ok_L\cong G_{\Ok_L}.$ We obtain an element $b\in G(L)$ which induces the $F$ crystal structure on $\D(\pdiv).$  By \cite[Lemma 1.1.12]{Ki1}, the filtration on $\D(\pdiv)(\Fpbar)$ is indued by a $G$-valued cocharacter $\mu_0$. Let $F$ the filtration on $\D(\pdiv)$ induced by $\mu_0$ and let $U_{G}^\circ$ be the opposite unipotent defined by $\mu_0$.
	
	\begin{prop}\label{unipfilt}
		Let $\tilde{\pdiv}$ be a lifting of $\pdiv_x$ to $\Ok_L$, with corresponding filtration $F' \subset \D(\pdiv)$. Then $\lpdiv$ is $G_{\Ok_L}$-adapted if and only if $F'$ is induced by a $G_{\Ok_L}$-valued cocharacter $\mu$ which is $G_{\Ok_L}$-conjugate to $\mu_0$. Furthermore, $F' \subset \D(\pdiv)$ corresponds to a $G_{\Ok_L}$-adapted lift of $\pdiv$ if and only if $F' = uF$, with $u \in U^{\circ}_G(\Ok_L)$.
	\end{prop}
	
	\begin{proof}
		We start with the first claim. Let $\tilde{\pdiv}_S$ be the base change of $\tilde{\pdiv}$ along $\Ok_L\rightarrow S$. Then $\tilde{\pdiv}_S$ is a lift of $\tilde{\pdiv}$, and we have a isomorphism $\D(\tilde{\pdiv})(S)\cong\D(\pdiv)\otimes_{\Ok_L}S$ which is compatible with $\varphi$. Under this isomorphism, $s_{\alpha,0}$ correspond to tensors $\tilde{s}_{\alpha}\in \D(\tilde{\pdiv})(S)^\otimes$, whose stabilizer can be identified with the base change $G_S$ of $G$ to $S$. Now if the filtration is induced by a $G(\Ok_L)$-valued cocharacter, we have the $\tilde{s}_{\alpha}$ specialize to $s_{\alpha,0}\in\mathrm{Fil}^0\D(\pdiv)^\otimes$, hence $\tilde{\pdiv}$ is $G_{\Ok_L}$-adapted.
		
		Conversely suppose $\tilde{\pdiv}$ is $G_{\Ok_L}$-adapted.   By \cite[Lemma 1.1.9]{Ki2}, the filtration on $\D(\pdiv)$ corresponding to $\tilde{\pdiv}$  is given by a $G_{\Ok_L}$-valued cocharacter $\mu$ which is $G_{\Ok_L}$-conjugate to $\mu_0^{-1}$.
		
		We now prove the statement about the filtrations. If $F' = uF$ with $u$ as above, then $F'$ is induced the by the conjugate of $\mu_0$ by $u$. It remains to prove the other implication. Let $M$ denote the Levi-subgroup of $G$ defined by $\mu_0$, let $P$ denote the parabolic subgroup of $G$ defined by $F$, and let $U_G$ be the unipotent radical of $P$. Define $\mathcal{P}$ to be the parahoric subgroup of $G(L)$ consisting of $g \in G(\Ok_L)$ which reduce to $P(\overline{\F}_p)$ modulo $p$. By the Iwahori decomposition, we have $\mathcal{P}(\Ok_L)=\mathcal{U}_G^\circ(\Ok_L)M(\Ok_L)\mathcal{U}_G(\Ok_L)$, where $\mathcal{U}_G(\Ok_L)=U(L)\cap \mathcal{P}(\Ok_L)$ and $\mathcal{U}_G^\circ(\Ok_L)=U_G^\circ(L)\cap\mathcal{P}(\Ok_L)$. Thus since $\mathcal{U}_G^\circ(\Ok_L)=pU_G^\circ(\Ok_L)$ we have $g=up$ for $u\in pU_G^\circ(\Ok_L)$ and $p\in P(\Ok_L)$. Since $p$ fixes $F$, we have $gF=uF$. 
		The uniqueness of $u$ follows from the fact that $\mathcal{U}_G^\circ(\Ok_L)\cap P(\Ok_L)=1$.
		
	\end{proof}
	
	\subsection{} Denote by $J_b$ the algebraic group over $\mathbb{Q}_p$ whose points in a $\mathbb{Q}_p$-algebra $R$ are given by 
	$$J_b(R) = \{g \in G(R \otimes_{\mathbb{Q}_p} L) : b\sigma(g) = gb\}$$
	The $\Q_p$ points of $J_b$ can be identified with the subset of automorphisms of $\pdiv_x$ in the isogeny category, which fix the tensors $s_{\alpha,0}$. By work of Kottwitz \cite{Ko1}, $J_b$ is an inner form of $M$, the centralizer of $\overline{\nu}_b$ in $G$.
	
	We can now prove the main result of this section:
	
	\begin{thm}\label{canliftcrystal}
		Let $\pdiv$ be a $\mu$-ordinary $p$-divisible group with $G$-structure. Then there exists a unique $G_{\Ok_L}$-adapted lift 
		$\lpdiv$ to $\Ok_L$, such that the action of $J_b(\Q_p)$ on $\pdiv$ lifts (in the isogeny category) to $\lpdiv$.
	\end{thm}
	\begin{proof}
		By Proposition \ref{cats}, $b$ is $\sigma$-conjugate to $\sigma(\mu(p))$ by an element of $G(\Ok_L)$. Upon changing the isomorphism $$V_{\Z_p} \otimes_{\Z_p}\Ok_L\cong\D(\pdiv)$$ we may assume $b=\sigma(\mu(p))$. The filtration on $\D(\pdiv)(\Fpbar)$ is induced by the reduction of $\mu$ mod $p$. Consider the filtration induced by $\mu$ on $\D(\pdiv)$, and let $\lpdiv$ denote the corresponding lift. By Proposition \ref{unipfilt}, $\lpdiv$ is $G_{\Ok_L}$-adapted. 
		
		Recall that $M$ is the centraliser of $\overline{\mu}=\overline{\nu}_b=\nu_b$ in $G$. It is a Levi subgroup of $G$ defined over $\Z_p$. Let $h\in J_b(\Q_p)$, since $h^{-1}b\sigma(h)=b$, we have $h$ commutes with $\nu_b=\overline{\nu}_b$ and hence with $\overline{\mu}$. Thus $h\in M(L)$. By Lemma \ref{central}, $\mu$ is central in $M$, hence commutes with $h$. Thus the action of $J_b(\Q_p)$ respects filtrations and hence lifts to an action on $\tilde{\pdiv}$ (in the isogeny category) by Grothendieck-Messing theory.

		Finally we show uniqueness of $\tilde{\pdiv}$. Let $n$ be an integer such $n\nu_b$ is a genuine cocharacter of $G$ and let $\gamma_p=n\nu_b(p)$. Then $\gamma_p=b\sigma(b)...\sigma^{n-1}(b)\in J_b(\Q_p)$. We prove the stronger statement that $\tilde{\pdiv}$ is the unique lift for which the action of $\gamma_p$ extends.
		
		Let $\lpdiv'$ correspond to a different $G_{\Ok_L}$-adapted lifting of $\pdiv$, corresponding to a filtration $F'$ for which the action of $\gamma_p$ extends. By Proposition \ref{unipfilt}, the filtration is induced by the cocharacter $u\mu^{-1}u^{-1}$ for some $u\in U_G^{\circ}(\Ok_L)$ which reduces to the identity$\mod p$. 
		
		By Grothendieck Messing theory $\gamma_p$ preserves the filtration $F'\otimes_{\Ok_L}L\subset \D(\pdiv)\otimes_{\Ok_L}L$. Therefore $\gamma_pu\mu u^{-1}\gamma_p^{-1}$ defines the same parabolic in $G\otimes_{\Z_p}L$ as $u\mu u^{-1}$. Let $u'=\gamma_pu\gamma_p^{-1}$. Then since $\gamma_p$ commutes with $\mu$, we have $u'\mu u'^{-1}$ and $u\mu u^{-1}$ defines the same parabolic in $G\otimes_{\Z_p}L$, hence $u=u'$.
		
		By Lemma \ref{central}, $U_G^{\circ}$ is contained in the opposite unipotent $U^\circ_{G,\nu}$ defined by $\nu$. Therefore $\gamma_pu\gamma_p^{-1}\neq u$ unless $u=1$, in which case $F=F'$.
	\end{proof}
	
	\begin{definition}
		Let $\pdiv$ be a $\mu$-ordinary $p$-divisible group with extra structure. The lift constructed in Theorem \ref{canliftcrystal} is called the canonical lift of $\pdiv$ and will be denoted $\pdiv^{\can}$.
	\end{definition}
	
	Let us briefly explain the reason for the name canonical. In classical Serre--Tate theory for ordinary $p$-divisible groups, and in the PEL case considered by Moonen in \cite{Mo}, the canonical lift is constructed by taking the slope decomposition of the corresponding $p$-divisible group and showing there is a unique deformation of each constant slope part. The canonical lift corresponds to the $p$-divisible group which is the product of these deformations. In that case, if $M_{\Ok_L}$ denotes the Levi subgroup which preserves the slope decomposition, then the above can be stated group theoretically as saying that the canonical lift corresponds to the unique $M_{\Ok_L}$-adapted lifting, which is precisely how the canonical lift is contructed in Theorem \ref{canliftcrystal}. Moreover the canonical lift of an ordinary abelian variety is characterised as the unique lift for which all $\mod$ p automorphisms extend.
	
	\subsection{}We write $F^{\can}\subset \D(\pdiv)$ for the filtratioin corresponding to the canonical lift. Since $\mu$ and $\nu$ both factor through the same maximal torus,  we have that $F^{\can}$ breaks up into a direct sum corresponding to the slope decomposition. In other words, $F^{\can} = \displaystyle{\bigoplus_{i = 1}^r F^{\can}_i}$, with $F^{\can}_i \subset \D(\pdiv_{x,i})$. In particular the slope decomposition of the corresponding $F$-crystal with $G$ structure gives a slope decomposition $\pdiv=\prod_{i=1^r}\pdiv_{i}$, and this extends to a decomposition $\pdiv^{\can}=\prod_{i=1}^r\pdiv^{\can}_i$ of the underlying $p$-divisible group of $\pdiv^{\can}$.

	\section{Deformation theory}
	
	In this section we study the $G$-deformation space of a $\mu$-ordinary $p$-divisible group $\pdiv$ with $G$-structure, following \cite{Ki1}. By definition this is the subspace of the universal deformation space whose $\Ok_K$ points correspond to $G_{\Ok_L}$-adapted liftings as in the previous section. We show that the universal $p$-divisible group over the $G$-deformation space admits a slope filtration and the $G$-deformation space naturally has the structure of a  shifted subspace of a cascade.
	
	\subsection{}We retain the notation of the previous section so that $\pdiv$ is a $p$-divisible with $G$ structure over $\Fpbar$. We let $\mu$ be a cocharacter of $G_{\Ok_L}$ inducing the filtration on $\D(\pdiv)\otimes_{\Ok_L}\Fpbar$. Define $U^\circ\subset \GL(\D(\pdiv))$ to be the opposite unipotent of $\mu$ in $\GL(\D(\pdiv))$, and $U^\circ_G$ to be the opposite unipotent of $\mu$ in $G_{\Ok_L}$. 
	
	Let $R$ be the complete local ring at the identity section of $U^\circ$ and $R_G$ be the complete local ring at the identity section of $U^\circ_G$. The ring $R_G$ is a formally smooth quotient of $R$. We choose coordinates such that $R \tilde{\rightarrow} \Ok_L[|t_1, ... , t_N|]$ and $R_G \tilde{\rightarrow} R/(t_{n+1}, ...,t_N)$. 
	
	Define a Frobenius operator $\phi = \phi_R$ on $R$ (and $\phi = \phi_{R_G}$ on $R_G$) extending the one on $\Ok_L$, and sending the $t_i$ to $t_i^p$. Let $M = \D(\pdiv) \otimes R$, along with the filtration induced by $\mu$. Define a $\sigma$-semilinear Frobenius $\phi$ on $M$ given by the composite 
	$$\D(\pdiv)\otimes_{\Ok_L} R \xrightarrow{\phi \otimes \phi} (\D(\pdiv)\otimes_{\Ok_L} R) \xrightarrow{u} (\D(\pdiv)\otimes_{\Ok_L} R)$$
	where $u$ is the tautological $R$ point of $U^\circ$, given by the canonical inclusion of the coordinate ring of $U^\circ$ into $R$. This gives $M$ the structure of a filtered $F$-crystal over $R$ whose associated $p$-divisible group such is a versal deformation of $\pdiv$. The F-crystal structure on $\D(\pdiv) \otimes R$ is given by a connection $\nabla$. 
	
	This $p$-divisible group (and the crystal) pulls back to $R_G$ along the inclusion. The main result of \cite[\S1]{Ki1} is that an $\Ok_K$ point of $R$ factors through $R_G$ exactly when it is a $G_{\Ok_L}$-adapted lifting of $\pdiv$. The $F$-crystal on $R_G$ is given by $M_G = \D(\pdiv) \otimes_{\Ok_L} R_G$, along with the filtration induced by $\mu$, and Frobenius $\phi_G$ given by the composite 
	$$M_G = \D(\pdiv)\otimes_{\Ok_L} R_G \xrightarrow{\phi \otimes \phi} M_G \xrightarrow{u_G} M_G $$
	where $u_G$ is the tautological $R_G$ point of $U^\circ_G$.
	
	\subsection{}
	Now assume $\pdiv$ is a $\mu$-ordinary $p$-divisible group over $\Fpbar$. We fix the isomorphism $V\otimes_{\Z_p}\Ok_L\cong \D(\pdiv)$ such that $\varphi=b\sigma$ where $b=\sigma(\mu(p))$.
	
	Let $\nu_b=\overline{\nu}_b=\overline{\mu}$ be the Newton cocharacter of $\sigma(\mu(p))$ and let $U^{\circ}_{G,\nu}$ denote the opposite unipotent with respect to $\nu_b$. Recall we have the slope decomposition $\pdiv = \displaystyle \prod_{i=1}^{r} \pdiv_i$ and its extension to the canonical lift $\pdiv^{\can} = \displaystyle \prod_{i=1}^{r} \pdiv^{\can}_i$. Suppose that each $\pdiv_i$ is isoclinic of slope $\lambda_i$, with $\lambda_i > \lambda_j$ if $i<j$. For $1\le a \le b \le r$, let $\pdivab = \displaystyle \prod_{i=a}^b \pdiv_i$, and let $\pdivab^{\can}$ denote $\displaystyle \prod_{i=a}^{b} \pdiv^{\can}_i$. We write $\D_{(a,b)}$ for the corresponding summand of $\D(\pdiv)$, and $F^{\can}_{(a,b)}\subset \D_{(a,b)}$ for the filtration  corresponding to $\pdiv_{(a,b)}^{\can}$.
	
	The slope decomposition induces a slope filtration on $\pdiv$, given by \begin{equation}0\subset\pdiv_{(1,1)}\subset \pdiv_{(1,2)}\subset\hdots \subset \pdiv_{(1,r-1)}\subset \pdiv_{(1,r)}=\pdiv\end{equation}
	
	and the corresponding the filtration on $\D(\pdiv)$ is given by 
	
	\begin{equation}\label{filt}0\subset \D_{(r,r)}\subset \D_{(r-1,r)}\subset \hdots \subset \D_{(2,r)}\subset \D_{(1,r)}=\D(\pdiv)\end{equation}
	
	(Note the indexing is flipped since we are using the contravarient Dieudonn\'e module). By definition, elements in $U^\circ_{G,\nu}$ (and in the unipotent of $\nu$ in $\GL(\D(\pdiv))$) preserve the filtration \eqref{filt}, and act as the identity on the quotients $\D_{(a,r)}/\D_{(a+1,r)}.$

	\subsection{}We prove in the following proposition that the slope filtration lifts to $G_{\Ok_L}$-adpated deformations, with the associated graded pieces isomorphic to $\pdiv^{\can}_i$. More precisely:
	
	\begin{prop}\label{slopefilt}
		Let $\tilde{\pdiv}$ be a $G_{\Ok_L}$-adapted lifting to $\Ok_K$. Then, there exists a canonical filtration $0 \subset \tilde{\pdiv}_{(1,1)} \subset \tilde{\pdiv}_{(1,2)} ... \subset \tilde{\pdiv}_{(1,r)} = \tilde{\pdiv}$, such that $\tilde{\pdiv}_{(1,a)}$ is a deformation of $\displaystyle{\prod_{i=1}^{a} \pdiv_i \subset \pdiv}$. Further,  $\tilde{\pdiv}_{(1,a)}/ \tilde{\pdiv}_{(1,a-1)} = \pdiv^{\can}_a\times_{\Ok_L}\Ok_K$.
	\end{prop}
	\begin{proof}
		It suffices to prove that the slope filtration exists for the versal $G_{\Ok_L}$-adapted $p$-divisible group over $R_G$, and exhibits the required properties. In order to show the existence of the slope filtration, it suffices to show that $M_{(a,r)} = \D_{(a,r)}\otimes R_G$ with the filtration given by intersecting with $\mathrm{Fil}^1$ is a sub filtered F-crystal of $M_G$. It suffices to show that $M_{(a,r)}$ is stable under $\phi$ and $\nabla_G$. Showing that $M_{(a,r)}/M_{(a+1,r)}$ comes from base change from the $F$-crystal $\D_{(a,a)}$ would establish the last claim.  
		
		By \cite[E.1]{Ki2}, the connection $\nabla_G$ is given by an element of $\textrm{Lie}\ U_{G,\nu} \otimes \Omega^1_{R_G}$. It follows that $\nabla_G$ preserves all the $M_{(a,r)}$, and is the trivial connection on $M_{(a,r)}/M_{(a+1,r)}$.
		
		Lemma \ref{central} implies that $U^o_G \subset U^o_{G,\nu}$, and so $U_{G}^\circ$ also preserves the slope filtration on $\D(\pdiv)$, and acts as the identity on the associated graded. By definition, $\phi$ is given by composing the Frobenius in $\D(\pdiv)$ and the tautological $U^\circ_G$-point of $R_G$, therefore it preserves $M_{(a,r)}$ and the induced $\phi$ on $M_{(a,r)}/M_{(a+1,r)}$ is given by base extensions of the Frobenius on $\D_{(a,a)}$. The result follows.
	\end{proof}
	
	\subsection{}Let $\pdiv^{\uu}$ denote the $p$-divisible group over $\Spf R_G$. For $1\le a \le r$, define  $\pdiva^{\uu}$ to be the subgroup of $\pdiv^{\uu}$ coming from the slope filtration. For $1 \leq a \leq b \leq r$, we have $\pdivb^{\uu} \subset \pdiva^{\uu}$ and $\pdivab^{\uu} = \pdiv_{(1,b)}^{\uu}/\pdiv_{(1,a-1)}^{\uu}$ is a deformation of $\pdivab = \displaystyle{\prod_{i=a}^b \pdiv_i}$. 
	
	The unipotent group $U^\circ_G$ preserves the slope filtration of $\D(\pdiv)$, therefore, $U^\circ_G$ acts on the subquotient $\D(\pdivab)$. Let $\Uab$ be the image of $U^\circ_G$ in $\GL(\D(\pdivab))$. The group $\Uab$ is  smooth, and the map from $U^\circ_G$ to $\Uab$ is a smooth map. Moreover for $a\leq a'\leq b'\leq b$ the map $U_G^\circ\rightarrow U^\circ_{(a',b')}$ factors through $U_{(a,b)}^\circ$, hence we obtain smooth maps $U_{(a,b)}^\circ\rightarrow U_{(a',b')}^\circ$. Let $\Rab$ be the complete local ring of $\Uab^\circ$. We have an inclusion $\Rab \rightarrow R_G$ as well as maps $R_{(a',b')}\rightarrow R_{(a,b)}$ for any $a',b'$ as above. Having fixed $a$ and $b$, we can and do choose coordinates on $R$ and $R_G$ as above, such that $R_{(a,b)}=\Ok_L[t_1,\cdots, t_m]$.
	
	The filtered $F$-crystal $M_{(a,r)}/M_{(b,r)}$ is associated to the $p$-divisible group $\pdivab^{\uu}$ over $R_G$. Its underlying $R_G$ module can be identified with $\D_{(a,b)}\otimes_{\Ok_L}R_G$, where $\D_{(a,b)} = \D(\pdivab)$. We will need the following lemma:
	\begin{lemma}\label{descenttoRa}
		The filtered $F$-crystal $M_{(a,r)}/M_{(b,r)}$ comes from base change from a filtered $F$-crystal on $\D_{(a,b)}\otimes_{\Ok_L}R_{(a,b)}$. Moreover the slope filtration on $\pdiv_{(a,b)}^u$ descends to $R_{(a,b)}$.
	\end{lemma}
	\begin{proof}
		Recall that $R_G = \Ok_L[|t_1 \hdots t_m,\hdots t_n|]$ and $\Rab = \Ok_L[|t_1 \hdots t_m|]$. By definition, the image of $u_G$ (which we defined to be the tautological $\Spf R_G$ point of $U^\circ_G$) in $\GL(\D(\pdivab) \otimes R_G)$ factors through $\Uab^\circ$. By definition of $\Rab$, this point is induced by the tautological $\Rab$ point of $\Uab^\circ$. 
		
		Therefore, $\D_{(a,b)} \otimes R_G$ along with its Frobenius operator and filtration descends to $\Rab$. It suffices to prove that $\nabla$ also descends to $\Rab$. 
		
		The coordinates $t_i$ induce a Frobenius-equivariant map $R_G \rightarrow \Rab$, such that the composite $\Rab \rightarrow R_G \rightarrow \Rab$ is the identity on $\Rab$. Therefore, base changing $\D_{(a,b)}\otimes_{\Ok_L}R_G$ along $R_G\rightarrow R_{(a,b)},$ gives a filtered $F$-crystal over $\Rab$, whose underlying module is $\D_{(a,b)}\otimes_{\Ok_L}R_{(a,b)}$ in particular it is equipped with a Frobenius equivariant connection. 
		
		Base changing the filtered $F$-crystal $\D(\pdivab) \otimes \Rab$ to $R_G$, we get back $\D(\pdivab)\otimes R_G$ along with the original Frobenius and filtration. By \cite{integralMo}, the connection on a filtered F-crystal is determined by the rest of the data, therefore the connection base changed from $\Rab$ has to equal the original connection. The result follows.
	\end{proof}
	We use the same notation $\pdiv^{\uu}_{(a,b)}$ for the $p$-divisible group over $R_{(a,b)}$ constructed in the previous lemma. Since the slope filtration descend to $\pdiv^{\uu}_{(a,b)}$, for $a\leq a' \leq b' \leq b$ we have a subquotient over $\pdiv^{\uu}_{(a,b)}$, corresponding to $(a',b')$. By the same proof as the previous lemma, this descends to $R_{(a',b')}$. We define $\Def_{(a,b)}=\Spf R_{(a,b)}$ a smooth formal scheme. We obtain the following diagram:
	
	\begin{equation}\label{cascade}\xymatrix{&    &\Def_{(1,r)}\ar[ld]\ar[rd]& &\\
&\Def_{(1,r-1)} \ar[ld] \ar[rd]  & & \Def_{(2,r)}\ar[ld] \ar[rd] & \\
	 \Def_{(1,r-2),} & & \Def_{(2,r-1)}  & & \Def_{(3,r)}
}
	\end{equation}
	We will need an explicit description of the maps $\Def_{(1,r)}\rightarrow \Def_{(a,b)}$ on $\Ok_L$ points. Given a $G_{\Ok_L}$-adapted lifting of $\pdiv$ to $\Ok_L$, by Proposition \ref{unipfilt}  there exists a unique element $u\in U^\circ_G(\Ok_L)$ reducing to the identity mod $p$ such that the filtration on $\D(\pdiv)$ is induced by the translate of $F^{\can}$ by $u$. The image of the corresponding point in $\Defab$, corresponds to the filtration on $\D_{(a,b)}$ given by the translate of $F^{\can}_{(a,b)}$ by $\overline{u}$, where $\overline{u}$ is the image of $u$ in $U_{(a,b)}^{\circ}$.
	
	We record the above the in following proposition:
	
	\begin{prop} \label{rest}
		The formal schemes $\Defab$ are smooth, and for $a\leq a'\leq b'\leq b$ the restriction maps $\Defab \rightarrow \Def_{(a',b')}$ are smooth. The $\Ok_L$ points of $\Defab$ are in bijection with elements of $\Uab^\circ(\Ok_L)$ reducing to the identity mod $p$, where the bijection is given by sending $u\in \Uab^\circ(\Ok_L)$  to the filtration of $\D_{(a,b)}$ induced by $u\mu^{-1}$. Finally, $\Uab^\circ$ is stable under conjugation by the image of the projection of $\nu$ to $\GL(\D_{(a,b)})$. 
	\end{prop}
	\begin{proof}
		The only part that hasn't been proved is $\Uab^\circ$ being stable under conjugation by the projection of $\nu$. This follows directly from the fact that $U^\circ_G$ is stable by conjugation by the image of $\nu$.
	\end{proof}

We end this section by tabulating facts about $\Def_G$. To that end, consider the following hypotheses.	\subsection*{Hypotheses}
	Let $\Def_{(1,r)} \subset \Def(\pdiv)$ be a smooth subscheme satisfying the following properties:
	
	\begin{enumerate} 
		\item The slope filtration canonically lifts to the family $\pdiv^{\uu} \rightarrow \Def_{(1,r)}$. \label{hyp1}

		\item $\Defab$ is a formally smooth subscheme of $\Def(\pdivab)$ defined over $\Ok_L$, and contains $\pdivab^{(\can)}$. (Here, $\Defab$ denote the image of $\Def_{(1,r)}$ inside $\Def(\pdivab)$ given by the canonical restriction maps.) \label{hyp3}
		
		\item For $a \le b$, the restriction map $\Def_{(1,r)} \rightarrow \Defab$ is smooth. \label{hyp4}
		
		\item There exists a unipotent subgroup\footnote{For us, the group $U_G$ will take the place of $U$.} $U \subset \GL(\D)$ preserving the slope filtration, such that the $\Ok_L$ points of $\Def_{(1,r)}$ precisely correspond to filtrations induced by $g\mu$, with $g \in U(\Ok_L)$. Further, $U$ is preserved under conjugation by the image of $\nu$. \label{hyp5}
		
		\item The $\Ok_L$ points of $\Defab$ correspond to the filtrations of $\D(\pdivab)$ induced by $g\mu$ with $g\in \Uab^\circ(\Ok_L)$, where $\Uab^\circ \subset \GL(\D(\pdivab))$ is the image of $U$ restricted to $\D(\pdivab)$. Further, $\Uab^\circ$ is preserved under conjugation by the image of $\nu$. \label{hyp6}
	\end{enumerate}
	
	The following result follows directly from the definitions of the properties:
	\begin{prop}\label{in}
		If $\Def_{(1,r)}$, as above, satisfies all the above properties, then so do the spaces $\Defab$, for $1 \le a \le b \le r$. 
	\end{prop}
	
	\begin{prop}
		The $G$-deformation space $\Def_G$ of $\pdiv$ safisfies all the above properties. 
	\end{prop}
	\begin{proof}
		Property \ref{hyp1} follows from Proposition \ref{slopefilt}. Properties \ref{hyp3} -- \ref{hyp6} follow from Proposition \ref{rest}.
	\end{proof}

	\section{Cascades and Fontaine--Laffaille theory}
		For this section, notation is as in \S 4. We first recall the definition of a cascade from \cite[\S 2.2]{Mo}. 
	\begin{definition}
		An $r$-cascade over $\Spec \Ok_L$ consists of the following data: 
		\begin{enumerate}
			\item Commutative formal groups $E_{(i,j)}$ over $\Ok_L$.
			\item Objects $\Gamma_{(i,j)}$ for $1 \le i < j \le r$; if $i \geq j$, we put $\Gamma_{(i,j)} = \Spec \Ok_L$. 
			\item Morphisms $\lambda_{(i,j)}: \Gamma_{(i,j)} \rightarrow \Gamma_{(i,j-1)}$ and $\rho: \Gamma_{(i,j)} \rightarrow \Gamma_{(i+1,j)}$ satisfying the commutativity relation $\rho_{(i,j-1)} \circ \lambda_{(i,j)} = \lambda_{(i+1,j)} \circ \rho_{(i,j)}$.
			\item The structure on $\Gamma_{(i,j)}$ of a biextension of $(\Gamma_{(i,j-1)},\Gamma_{(i+1,j)})$ by $E_{(i,j)} \times \Gamma_{(i+1,j-1)}$.
		\end{enumerate}
	\end{definition}
	As in \cite{Mo}, we note that the fiber of $\Gamma_{(i,j)}$ over any point of $\Gamma_{(i+1,j)}$ ({\it resp.} $\Gamma_{(i,j-1)}$) is a formal group. Further, every $r$-cascade has a natural zero section $0 \in \Gamma_{(1,r)}(\Ok_L)$ (denote by the same symbol $0$ its image in $\Gamma_{(i,j)}$). By choosing $E_{(i,j)} = \Ext(\pdiv^{\can}_j,\pdiv^{\can}_i)$ and $\Gamma_{(i,j)}(\pdiv^{\can})$ to be the subspace of the deformation space of $\pdiv_{(i,j)}$ which admits a slope filtration as in Proposition \ref{slopefilt}, we get a natural $r$-cascade which contains $\Spf R_G$, where $r$ is the number of slopes in the slope decomposition. This follows from the same proof as in \cite[\S 2.3.6]{Mo}. In this case, the canonical lift is the $0$-section of $\Gamma(\pdiv^{\can})$.

		\begin{definition}
		Let $\Gamma = \Gamma_{(i,j)}$ be an $r$-cascade and let $\Delta$ be a formal subscheme. Let $\Delta_{(i,j)}$ denote the image of $\Delta$ in $\Gamma_{(i,j)}$. We say that $\Delta$ is a shifted subcascade if:
		\begin{enumerate}
			\item The zero section belongs to $\Delta$. 
			\item The fiber of $\Delta_{(i,j)}$ over $0\in \Delta_{(i+1,j)}$ (and $0\in \Delta_{(i,j-1)}$) is a subgroup of the fiber of $\Gamma_{(i,j)}$ over $0 \in \Gamma_{(i+1,j)}$ (and $0\in \Gamma_{(i,j-1)}$).
			\item The fiber of $\Delta_{(i,j)}$ over any other $\Ok_K$ point $P\in \Delta_{(i+1,j)}(\Ok_K)$ (and $P\in \Delta_{(i,j-1)}(\Ok_K)$) is a coset of a subgroup of the fiber of $\Gamma_{(i,j)}$ over $P \in \Gamma_{(i+1,j)}(\Ok_K)$ (and $P\in \Gamma_{(i,j-1)}(\Ok_K)$).
		\end{enumerate}
	\end{definition}
	
	The main result of this section is 
	\begin{thm}\label{cmhyp}
		If $\Def_{(1,r)} \subset \Def(\pdiv)$ satistfies Properties \ref{hyp1} -- \ref{hyp6}, then $\Def_{(1,r)}$ is a shifted subcascade of the cascade $\Gamma(\pdiv)$. Consequently, $\Def_G$ is a shifted subcascade. 
	\end{thm} 
	
	\begin{proof}
		The space $\Def_{(1,r)}$ contains the $0$ section of $\Gamma(\pdiv)$, beause by Hypothesis \ref{hyp5} it contains the canonical lift. 
		
		The proof now follows from Corollary \ref{twoslopesgroup} and Proposition \ref{FLcosetcalc} below. 
		
	\end{proof}
	
	We first recall some definitions and results from \cite{FL}.

	\begin{definition}A filtered Dieudonn\'e module consists of a triple $(M,(M_{i})_{i\in\Z},(\varphi_i)_{i\in\Z})$ where:

		i) $M$ is a finite $\Ok_L$-module.
		
		ii) $M_i\subset M$ is a decreasing filtration which is exhaustive and separated.
		
		iii) $\varphi_i:M_i\rightarrow M$ a $\sigma$-semilinear map such that the following diagram commutes:
		
\[\xymatrix{M_{i+1} \ar[r]	\ar[d]_{\varphi_{i+1}} & M_i\ar[d]^{\varphi_i}\\
	M \ar[r]^{\times p} &M
}\]
		
		We write $MF$ for the category of filtered Dieudonn\'e modules and $MF^{tor}$ the subcatory of $MF$  consisting of triples above
		for which $M $ is a torsion $\Ok_L$ module. It is proved in \cite{FL} that $MF^{\tor}$ is an abelian category. As a special case of the above definition we have the following:
		
	\end{definition}
	
	\begin{definition}A Honda system is a pair $(\D,F)$ consisting of a finite free $\mathcal{O}_L$-module $\D$, a $\mathcal{O}_L$ submodule and a $\sigma$-semilinear map $\varphi:\D\rightarrow \D$ such that:
		
		i) $p\D\subset \varphi(\D)$.
		
		ii) The natural map induces an isomorphism $F/pF\cong \D/\varphi(\D)$.
	\end{definition}
	
	It is known that there is an anti-equivalence of categories between the category of $p$-divisible groups over $\mathcal{O}_L$ and the category of Honda systems. The equivalence is given by sending a $p$-divisible group over $\mathcal{O}_L$ to the Dieudonn\'e module of its special fiber together with the lift of the Hodge filtration (see \cite[\S 5 Theorem 1.6]{Me}).
	
	Suppose $\pdiv$ is a $\mu$-ordinary $p$-divisible group with $G$-structure whose underlying $p$-divisible has two slopes. Let $\pdiv=\pdiv_1\times\pdiv_2$ be given by the slope decomposition and let $(\D_i,F_i)$ denote the corresponding Honda systems for the $p$-divisible groups $\mathscr{G}_1^{\can}$, $\mathscr{G}_2^{\can}$. Suppose $\mathscr{G},\mathscr{G}'$ are two extensions of $\mathscr{G}^{\can}_2$ by $\mathscr{G}^{\can}_1$ with corresponding Honda system $(M,L)$  and $(M,L')$. We would like to compute the Honda system corresponding to the sum $\mathscr{G}''$ of $\mathscr{G}_2^{\can}$ and $\mathscr{G}_1^{\can}$ in $\Ext(\pdiv^{\can}_2,\pdiv^{\can}_1)$.
	
	Recall that in an arbitrary abelian category, given two extensions:
	$$0\rightarrow A\rightarrow B\rightarrow C\rightarrow 0$$
	$$0\rightarrow A\rightarrow B'\rightarrow C\rightarrow 0$$
	the sum of these classes in $\Ext^1(A,C)$ can be computed in the following way. First form the pullback of 
	$$0\rightarrow A\oplus A\rightarrow B\oplus B'\rightarrow C\oplus C\rightarrow 0$$ by the diagonal morphism $\Delta:C\rightarrow C\oplus C'$ to get obtain an extension $$0\rightarrow A\oplus A \rightarrow D\rightarrow C\rightarrow 0$$ We then push out this extension along the addition map $\Sigma: A\oplus A\rightarrow A$ to get $$0\rightarrow A\rightarrow E\rightarrow C\rightarrow 0$$ which is gives the sum of the two extensions.
	
	If $A$ and $C$ are $p$-divisible groups and $B$ is an extension of $A$ and $C$ in the category of fppf sheaves over $\Ok_L$, then by \cite[I 2.4.3]{Me} $B$ is a $p$-divisible group. Moreover if we compute the sum of the extensions $B$ and $B'$ in $\Ext^1(A,C)$ in the above way, at each stage, the extensions $D$ and $E$ obtained are also $p$-divisible groups. Thus to compute the Baer sum of extensions of $p$-divisible groups, we may work completely in the category of $p$-divisible groups by using the above pullback and pushout constructions. By the properties of the above equivalence, to compete the sum of the corresponding extensions of Honda systems, we may apply the corresponding construction in the category of Honda systems.
	
	\begin{remark}
		The category of Honda systems does not have pushouts in general. Strictly speaking, to carry out the constructions, one should reduce mod $p^n$ to get an object of $MF^{tor}$ which is an abelian category, apply the pullback-pushfoward construction, then take an inverse limit to recover the Honda system. On the $p$-divisible groups side, this corresponds to taking $p^n$ torsion, then taking a direct limit.
	\end{remark}
	
	\subsection*{Application to $p$-divisible groups}
	Recall that $\pdiv^{\can}=\displaystyle{\prod_{i=1}^r \pdiv^{\can}_i}$ is the canonical lift of a $\mu$-ordrinary $p$-divisible group $\pdiv_x = \displaystyle{\prod_{i=1}^r \pdiv_i}$. The associated Honda-system breaks up into a direct sum $\displaystyle{\big(\bigoplus_{i=1}^r \D_i,\bigoplus_{i=1}^r F_i\big)}$. Note that a choice of cocharcater $\mu$ as in \S 4.1 fixes direct summands of the $F_i$ in $\D_i$. To that end, let $F'_i\subset \D_i$ be such that $\D_i = F_i \oplus F'_i$. 
	
	Let $\tilde{\pdiv}$ be a $G_{\Ok_L}$-adapted lifting of $\pdiv_x$, then with associated Honda system ($\D,F$), then since $\tilde{\pdiv}$ and $\pdiv^{\can}$ have the same special fibre, we have an identification of $\D$ and $\displaystyle{\bigoplus_i \D_i}$, compatible with Frobenius.
	
	\subsection{}
	We fix $1\le a < b \le r$. For $K/L$ a finite extension, let $s$ an $\Ok_K$-valued point in $\Defab$, corresponding to the deformation $\pdivh_{(a,b)}$ of $\pdivab$. We define $\Def_{(a-1,b)}^s$ to be the fiber of $\Def_{(a-1,b)}\rightarrow\Def_{(a,b)}$ over $s$, and $\Def_{(a,b+1)}^s$ to be the fiber of $\Def_{a,b+1} \rightarrow \Def_{(a,b)}$ over $s$. Then $\Def_{(a-1,b)}^s$ is a subspace of $\Ext(\pdivh_{(a,b)},\pdiv^{\can}_{a-1}\otimes\Ok_K)$, and the analogous statement holds for $\Def_{(a,b+1)}^s$. Recall that we also have a canonical surjective map $U^\circ_{(a-1,b)} \rightarrow U^\circ_{(a,b)}$, as well as a canonical map $\displaystyle U^\circ_{(a-1,b)} \rightarrow \Hom(F_{a-1}, \D_{(a,b)})$. The latter map is described by $u \mapsto u_{a-1} = (u-1)|_{F_{a-1}}$, where $u-1$ refers to the endomorphism obtained by treating $u$ and $1$ as elements in $\displaystyle \End(\D_{(a-1,b)})$. Note that this map is a group homomorphism since $U^\circ_{(a,b)}$ is a subquotient of $U^\circ$ and hence is an abelian group. We remark that $u_{a-1}$ is actually an element of $\displaystyle \Hom(F_{a-1}, \bigoplus_{i=a}^b F'_i)$.
	
	Suppose we are given $\tilde{\pdiv}$ and $\tilde{\pdiv}'$, $\Ok_L$-valued points in $\Def_{(a-1,b)}^0$. We assume the filtration on $\D(\pdiv_{(a-1,b)})$ corresponding to $\tilde{\pdiv}$ ({\it resp.} $\tilde{\pdiv}'$) are given by elements $u$ ({\it resp.} $u'$) in $U^\circ_{(a-1,b)}(\Ok_L)$ as in Hypothesis \ref{hyp6}. 
	
	Suppose $\tilde{\pdiv}''$ is the Baer sum of $\tilde{\pdiv}$ and $\tilde{\pdiv}'$ in $\mathscr{E}^1(\pdivab^{\can},\pdiv_{a-1}^{\can})$.
	
	\begin{prop}\label{twoslopecalc}
		The filtration on $\D$ corresponding to $\lpdiv''$ is induced by the cocharacter which is the conjugate of $\mu$ by $u.u'\in U^{\circ}_{(a-1,b)}(\Ok_L)$. 
	\end{prop}
	\begin{proof}
		
		By Lemma \ref{central}, we have $U_G\subset U_{G,\nu}\subset U_{\GL,\nu}$. Hence by Lemma \ref{unipfilt}, the filtration $F$ corresponding to $\lpdiv$ is generated by $F_i$ for $a\leq i \leq b$, as well as the elements $ f_{a-1} + u_{a-1}(f_{a-1})$, for all $f_{a-1}\in F_{a-1}$. Similarly, the filtration $F'$ corresponding to $\lpdiv'$ is generated by $F_i$ for $a\leq i\leq b$, as well as the elements $f_{a-1} + u'_{a-1}(f_{a-1})$, for all elements $f_{a-1}\in F_{a-1}$. Recall that $u,u' \in U_{a-1,b}$, and these elements equal the identity when restricted to $\D_{(a,b)}$. 

		
		We apply the pullback-pushforward construction to compute the Baer sum. The pushforward of the extension  $\D_{(a-1,b)}\oplus \D_{(a-1,b)}$ by the $\Sigma:\D_{(a,b)}\rightarrow \D_{(a,b)}\oplus \D_{(a,b)}$ is given by the Honda system whose underlying $F$-crystal is $(\D_{a-1}\oplus \D_{a-1}\oplus \D_{(a,b)},\varphi_{a-1},\varphi_{a-1},\varphi_{(a,b)})$, and the filtration is generated by $F_i$ for $i\leq a\leq b$, $ f_{a-1}^{(1)} + u_{a-1}(f_{a-1})$ and $ f_{a-1}^{(2)} + u_{a-1}(f_{a-1})$ for $f_{a-1}\in F_{a-1}$. Here, the superscripts (1) and (2) denote which copy of $D_{a-1}$ the $f_{a-1}$ lie in.
		
		
		Then pulling back along the diagonal, we obtain the Honda system with underlying $F$-crystal $(\D_{a-1}\oplus \D_{(a,b)}).$ The filtration in this case is generated by the $F_i$ for $a\leq i\leq b$ as well as the $f_{a-1} + (u_{a-1} + u'_{a-1})(f_{a-1})$ for $f_{a-1}\in \D_{a-1}$. This filtration is visibly seen to be induced by the translate of $F_1\oplus F_2$ by $u.u'$, which is clearly an element of $U^{\circ}_{(a-1,b)}$. The lemma follows.
	\end{proof}
	
	\begin{cor}\label{twoslopesgroup}
		The space $\Def_{(a-1,b)}^0$ is a subgroup of $\Ext(\pdivab,\pdiv_{(a-1)})$, and the analogous statement holds for $\Def_{(a,b+1)}^0$. 
	\end{cor}
	\begin{proof}
		The above computation shows that the $\Ok_L$ points of $\Def_{(a-1,b)}^0$ is a subgroup. This space is also a smooth formal scheme over $\Spec \Ok_L$ (by Hypothesis \ref{hyp3}), therefore the set of $\Ok_L$ points are dense. The result follows for $\Def_{(a-1,b)}$. 
		
		The same computation as in Proposition \ref{twoslopecalc} shows that the $\Ok_L$ points of $\Def_{(a,b+1)}^0$ is a subgroup, and the above argument applies to finish the proof of the corollary.
	\end{proof} 
	
	\subsection{}We now treat the case when $s \neq 0$. Recall that $s$ corresponds to the deformation $\pdivh_{(a,b)}$ of $\pdivab$. 
	
	\begin{prop}\label{FLcosetcalc}
		$\Def_{(a-1,b)}^s$ is the translate of a subgroup of $\Ext(\pdivh_{(a,b)},\pdiv^{\can}_{a-1}\otimes\Ok_K)$, and $\Def_{(a,b+1)}^s$ is the translate of a subgroup of $\Ext(\pdiv^{\can}_{b+1},\pdivh_{(a,b)})$.
	\end{prop}
	\begin{proof}
		We will prove the result only in the case of $\Def_{(a-1,b)}^s$, as the second case follows identically. 
		The argument in Corollary \ref{twoslopesgroup} applies to reduce the proof to the case of $\Ok_L$-points, i.e. we may assume that $s$ is an $\Ok_L$-valued point, and it suffices to check that the $\Ok_L$-valued points in the fiber over $s$ form a coset. To that end, let $z$ be an $\Ok_L$ point of $\Def_{(a-1,b)}^s$, we must check for any $x,y\in\Def_{(a-1,b)}^s(\Ok_L)$, that $x+y-z\in\Def_{(a-1,b)}^s(\Ok_L)$. Let $u^x,u^y,u^z \in U_{(a-1,b)}(\Ok_L)$ denote the elements corresponding to the points $x,y,z$. The fact that $x,y,z \in \Def_{(a-1,b)}^s $ has the consequence that $u^x,u^y,u^z$ map to the same element $u^{a,b}\in U_{(a,b)}$. 
		
		
		A calculation entirely analogous to the one carried out in Proposition \ref{twoslopecalc} yields that the Honda system corresponding to the Baer sum of $x$ and $y$ (thought of as elements in $\Ext(\pdivh_{(a,b)},\pdiv^{\can}_{a-1}\otimes\Ok_K))$ has the following form: the underlying Dieudonne module is the same, and the filtration is spanned by $F_{(a,b)} \subset \D_{(a,b)}$ and $f_{a-1} +(u_{a-1}^{x} + u^y_{a-1})(f_{a-1})$ for all $f_{a-1}\in F_{a-1}$. Here, $F_{(a,b)}\subset \D_{(a,b)}$ is the filtration corresponding to the point $s$ and the element $U^{a,b}$. Note that this $p$-divisible group isn't necessarily an element of $\Def_{(a-1,b)}^s$. Indeed, it is possible that there is no element of $U_{(a-1,b)}(\Ok_L)$ that simultaneously restricts to $u^{a,b} \in U_{(a,b)}$, as well as maps to $u^x_{a-1} + u^y_{a-1} \in \Hom(F_{a-1},\D_{(a,b)})$. 
		
		However, the element of $\Def_{(a-1,b)}^s$ corresponding to $x + y - z$ has filtration equal to the span of $F_{(a,b)}\subset \D_{(a,b)}$, and $f_{a-1} + (u_{a-1}^x + u_{a-1}^y - u_{a-1}^z)(f_{a-1})$ for $f_{a-1}\in F_{a-1}$. It is easy to see that there indeed is an element of $U_{(a-1,b)}$ that restricts to $u^{a,b} \in U_{(a,b)}$, and maps to $u_{a-1}^x + u_{a-1}^y - u_{a-1}^z \in \Hom(F_{a-1},\D_{(a,b)})$, namely the element $u_x.u_y.(u_z)^{-1}$. Therefore, $x+y-z$ does indeed lie in $\Def_{a-1,b}^s$ as required.

	\end{proof}

	\section{Applications to Shimura Varieties}
	
	\subsection{}We will first give a brief summary of the definitions and constructions of Shimura varieties of Hodge-type and their integral models, mainly referring to \cite{Ki2} and \cite{Del} for details and further background. Let $G$ be a connected reductive group over $\mathbb{Q}$ which is unramified over $\mathbb{Q}_p$ - it follows that there exists a reductive group $G_{\mathbb{Z}_{(p)}}$ over $\mathbb{Z}_{(p)}$ with generic fiber $G$.
	
	Let $X$ be a conjugacy class of maps of algebraic groups over $\mathbb{R}$ 
	$$ h: \mathbb{S} = \Res_{\mathbb{C}/\mathbb{R}}\mathbb{G}_m \rightarrow G_{\mathbb{R}}, $$
	such that $(G,X)$ is a Shimura datum (see \cite[\S 2.1]{Del}).  Let $c$ be complex conjugation, then $\text{Res}_{\C/\mathbb{R}}(\C)\cong(\C\otimes_{\R}\C)^\times\cong \C^\times \times c^*(\C^\times)$ and we write $\mu_h$ for the cocharacter given by $$\C^\times\rightarrow \C^\times\times c^*(\C^\times)\xrightarrow h G(\C)$$ We let $E = E(G,X)$ be the minimal field of definition of the conjugacy class of $\mu_h$ (\cite[II, \S 4.5]{Mi}). Fix a compact open subgroup $\mathrm{K}\subset G(\A_f)$. Then by work of Shimura, Deligne, Milne, et al one can associate to this data a variety $\mathrm{Sh_K}(G,X)$ which is defined over $E$. 
	
	\subsection{} 
	We assume that there is an embedding of Shimura data
	$$ (G,X) \hookrightarrow (\GSp ,S^{\pm}). $$
	Here, $\GSp$ is the group of symplectic similitudes on a $\Q$-vector space $V$ with a perfect alternating form $\psi$, and $S^{\pm}$ is Siegel double space. For a sufficiently small compact open $\mathrm{K}'\subset \GSp(\A_f)$, the Shimura variety $\mathrm{Sh_{K'}}(\GSp,S^{\pm})$ is the moduli space of polarized abelian varieties along with some level structure depending on $\mathrm{K}'$.
	
	As in \cite[\S 1.3.3]{Ki2}, we may assume the embedding Shimura data is induced by an embedding $G_{\mathbb{Z}_{(p)}} \hookrightarrow \GSp(V_{\mathbb{Z}_{(p)}})$ for some lattice $\mathbb{Z}_{(p)}$ which can be assumed to be self dual with respect to $\psi$. Henceforth, we assume that $\mathrm{K}=\mathrm{K}_p\mathrm{K}^p$ where $\mathrm{K}_p = G_{\mathbb{Z}_{(p)}}(\mathbb{Z}_p)$ and $\mathrm{K}^p \subset G(\mathbb{A}^p_f)$ is sufficiently small. We assume $\mathrm{K}'=\mathrm{K}'_p\mathrm{K}'^p$, where $\mathrm{K}'_p$ is the stabilizer in $\GSp(\Q_p)$ of $V_{\Z_p}$ and  $\mathrm{K}'_p \subset \GSp(\A_f^p)$ is a compact open which contains $\mathrm{K}^p$, It is shown in \cite[Lemma 2.1.2]{Ki1} that for sufficiently small $\mathrm{K}'^p$ there is a closed embedding:
	$$\mathrm{Sh_K}(G,X) \hookrightarrow \mathrm{Sh_{K'}}(\GSp,S^{\pm})$$

	The variety $\mathrm{Sh_{K'}}(\GSp,S^{\pm})$ exntends to a smooth scheme $\mathscr{S}_{\mathrm{K}'}(\GSp,S^{\pm})$ over $\mathbb{Z}_{(p)}$ by extending the moduli problem to schemes over $\Z_{(p)}$. Let $\Ok_E$ be the ring of integers of $E$, and $\Ok_{E_{(p)}}$ be its localization at some prime (which we now fix) over $p$. We denote by $\mathscr{S}_{\mathrm{K}}(G,X)$ the normalization of the closure of $\mathrm{Sh_K}(G,X)$ in $\mathscr{S}_{\mathrm{K}'}(\GSp,S^{\pm}) \otimes_{\mathbb{Z}_{(p)}}\Ok_{E_{(p)}}$. Then by the main result of \cite[\S 2]{Ki1} we have that  $\mathscr{S}_{\mathrm{K}}(G,X)$ is smooth over $\Ok_{E_{(p)}}$. 
	
	\subsection{} We need a more explicit description of the formal neighbourhood of a point $x\in\mathscr{S}_{\mathrm{K}}(G,X)$. To do this, we need to introduce Hodge cycles.
	By \cite[1.3.2]{Ki1}, the subgroup $G_{\Z_{(p)}}$ is the stabilizer of a collection of tensors $s_\alpha\in V_{\Z_{(p)}}^\otimes$. Let $h:\mathcal{A}\rightarrow \mathscr{S}_{\mathrm{K}}(G,X)$ denote the pullback of the universal abelian variety on $\mathscr{S}_{\mathrm{K}'}(\GSp,S^\pm)$ and let $V_B:=R^1h_{\an,*}\Z_{(p)}$, with $h_{\an}$ is the map of complex analytic spaces associated to $h$. We also let $\mathcal{V}=R^1h_*\Omega^\bullet$ be the relative de Rham cohomology of $\mathcal{A}$. Using the de Rham isomorphism, the $s_\alpha$ give rise to a collection of Hodge cycles $s_{\alpha,dR}\in \mathcal{V}_\C^\otimes$, where $\mathcal{V}_\C$ is the complex analytic vector bundle associated to $\mathcal{V}$. By \cite[\S 2.2, 2.3]{Ki1}, these tensors are defined over $\mathcal{O}_{E,(p)}$.
	
	Similarly for a finite prime $l\neq p$, we let $\mathcal{V}_l=R^1h_{\et*}\Q_l$ and $\mathcal{V}_p=R^1h_{\eta,\et*}\Z_p$ where $h_\eta$ is the generic fibre of $h$. Using the \'etale-Betti comparison isomorphism, we obtain tensors $s_{\alpha,l}\in \mathcal{V}^\otimes_l$ and $s_{\alpha,p}\in\mathcal{V}_p^\otimes$. For $*=B, dR,l$ and $x\in \mathscr{S}_{\mathrm{K}}(G,X)(T)$, we write $\mathcal{A}_x$ for the pullback of $\mathcal{A}$ to $x$, $\pdiv_x$ for the associated $p$-divisible group and $s_{\alpha,*,x}$ for the pullback of $s_{\alpha,*}$ to $x$.

	Let $x\in\mathscr{S}_{\mathrm{K}}(G,X)(\Fpbar)$ and $\tilde{x}\in\mathscr{S}_{\mathrm{K}_p}(G,X)(\Ok_K)$ be a lift of $x$. 
	The Tate module $T_p\pdiv_{\tilde{x}}$ is equipped with Galois invariant tensors $s_{\alpha,p,\tilde{x}}\in (T_p\pdiv_{\tilde{x}})^{\otimes}$ whose stabiliser can be canonically identified with $G_{\Z_p}$.
	Let $\D(\pdiv_x)$ denote the Dieudonn\'e module of $\pdiv_x$, and let $s_{\alpha,0,x}\in(\D(\pdiv_x)\otimes_{\Z_p}\Q_p)^\otimes$ denote the image of $s_{\alpha,p,x}$ under the $p$-adic comparison isomorphism; by definition these are $\varphi$-invariant. It is shown in \cite{Ki2} that $s_{\alpha,0,x}$ are independent of the lift $\tilde{x}$. Moreover, we have $s_{\alpha,0,x}\in\D(\pdiv_x)^\otimes$ and there is an isomorphism of $\Ok_L$-modules:
	$$V_{\Z_p}^{\vee}\otimes_{\Z_p}\Ok_L\cong \D(\pdiv_x)$$
	which takes $s_{\alpha}$ to $s_{\alpha,0,x}$. In the notation of \S 3, we have $\pdiv_x$ is a $p$-divisible group with $G$-structure on $V_{\Z_p}$. We let $G_{\Ok_L}\subset GL(\D(\pdiv_x))$ denote the stabilizer of the $s_{\alpha,0,x}$. Then fixing an isomorphism $V_{\Z_p}^{\vee}\otimes_{\Z_p}\Ok_L\cong \D(\pdiv_x)$ gives a identification $G_{\Ok_L}\cong G\otimes_{\Z_p}\Ok_L$. Recall we may associate to this data the subspace $\Spf R_G$ of the universal deformation space of $\pdiv_x$ which cuts out the $G_{\Ok_L}$-adapted deformations.

	Let $\widehat{U}_x$ denote the formal neighbourhood of $\mathscr{S}_{\mathrm{K}}(G,X)$ along $x$. The following result of Kisin describes a formal neighbourhood of the Shimura variety at $x$.
	
	\begin{prop}\cite[Proposition 1.1.10]{Ki2}
		We have an identification $$\Spf R_G\cong \widehat{U}_x$$ which takes the universal $p$-divisible group on $\widehat{U}_x$ to the one $\Spf R_G$.
	\end{prop}

	\subsection{}We now fix a point $x \in \mathscr{S}_{\mathrm{K}}(G,X)(\Fpbar)$, and an identification of $\D(\pdiv_x)$ with $V_{\Z_p}^{\vee}\otimes_{\Z_p}\Ok_L$ respecting tensors. Let, we have $\varphi=b\sigma$ for some $b\in G(L)$. Recall that $J_b$ is the algebraic group over $\mathbb{Q}_p$ whose points in a $\mathbb{Q}_p$-algebra $R$ are given by 
	$$J_b(R) = \{g \in G(R \otimes_{\mathbb{Q}_p} L) : b\sigma(g) = gb\}$$
	
	Let $\Aut_{\mathbb{Q}}(\mathcal{A}_x)$ be the algebraic group defined over $\mathbb{Q}$ whose points in a $\mathbb{Q}$-algebra $R$ are given by 
	$$\Aut_{\mathbb{Q}}(\mathcal{A}_x)(R) = (\End_{\mathbb{Q}}(\mathcal{A}_x) \otimes R) ^{\times} $$
	Let $I_x \subset \Aut_{\mathbb{Q}}(\mathcal{A}_x)$ denote the closed subgroup consisting of those points which fix the tensors $s_{\alpha ,0,x}$ and $s_{\alpha,l,x}$ for all primes $l \neq p$. We have that $I_x \otimes \Q_p$ is (non-canonically) a subgroup of $J_b$. 
	
	Suppose $x$ is defined $\F_{p^n}$. Let $\gamma\in I_x(\Q)$ denote the $p^n$-Frobenius, then $\gamma$ maps to $\gamma_p:=b\sigma(b)\hdots \sigma^{n-1}(b)\in J_b(\Q_p)$.
	
	Fix a maximal torus $T\subset G$ defined over $\Z_p$, and $B\supset T$ a Borel subgroup. Let $\mu\in X_*(T)$ denote the dominant representative of the conjugacy class of $\mu_h^{-1}$ in $X_*(T)$. It is shown in \cite[Lemma 1.1.12]{Ki2} that $$b\in G(\Ok_L)\sigma(\mu(p))G(\Ok_L)$$
	
	By our assumptions $1\in X^G_{\sigma(\mu)}(b)$ so that $\sigma^{-1}(b)\in X^G_\mu(b)$, hence we have $\kappa_G(b)=[\mu]$ and $\overline{\nu}_b\leq\mu$.
	
	\begin{definition}A point $x\in \mathscr{S}_{\mathrm{K}}(G,X)(\overline{\mathbb{F}}_p)$ lies in the $\mu$-ordinary locus if the associated $\pdiv_x$ with $G$ structure is $\mu$-ordinary.
	\end{definition}
	We have the following results about the $\mu$-ordinary locus of $\mathscr{S}_{\mathrm{K}}(G,X)$. 
	
	\begin{prop}[Wedhorn\footnote{in the PEL type case}, Wortman] \label{Wort}
		The $\mu$-ordinary locus is an open dense subscheme of the special fiber of $\mathscr{S}_{\mathrm{K}}(G,X)$.
	\end{prop}
	\begin{proof}
		This is Theorem 1.1 of \cite{Wo}. 
	\end{proof}
	
	\begin{prop}\label{groupsisom}
		Let $x$ and $y$ be $\mu$-ordinary points in $\mathscr{S}_{K}(G,X)(\Fpbar)$. Then $\pdiv_x$ and $\pdiv_y$ are isomorphic $p$-divisible groups. 
	\end{prop}
	\begin{proof}
		By Corollary \ref{cortwo}, the $F$-crystals $\D(\pdiv_x)$ and $\D(\pdiv_y)$ are isomorphic. The result follows. 
	\end{proof}
	
	\subsection*{The Serre--Tate canonical lift}
	We now prove that every point in the $\mu$-ordinary locus of the Shimura variety has a canonical CM lift. 
	
	\begin{thm}\label{canlift}
		Let $x\in\mathscr{S}_{\mathrm{K}_p}(G,X)(\Fpbar)$ lie in the $\mu$-ordinary locus. Then there exists a unique lift $\tilde{x}\in\mathscr{S}_{\mathrm{K}_p}(G,X)(\Ok_L)$ such that the action of $I_x(\Q_p)$ on $\mathcal{A}_x$ lifts (in the isogeny category) to $\mathcal{A}_{\tilde{x}}$. Moreover $\tilde{x}$ is a special point of $\mathscr{S}_{\mathrm{K}_p}(G,X)(L)$.
	\end{thm}
	\begin{proof}
		The lift constructed in Theorem \ref{canliftcrystal} is $G_{\Ok_L}$-adapted, and therefore corresponds to a point $\tilde{x}$, of $\mathscr{S}_{\mathrm{K}_p}(G,X)(\Ok_L)$. We prove that $\tilde{x}$ satisfies the required properties. 
		
		By Theorem \ref{canliftcrystal}, the action of $J_b(\Q_p)$ lifts to $\pdiv_{\tilde{x}}$. Since $I_x(\mathbb{Q}) \subset J_b(\mathbb{Q}_p)$, it follows by the theorem of Serre and Tate that the action of $I_x(\mathbb{Q})$ lifts (in the isogeny category) to $\mathcal{A}_{\tilde{x}}$.
		
		We now show that $\tilde{x}$ is actually a special point. Since $I_x$ fixes the tensors $s_{\alpha, 0,x}$, it also fixes $s_{\text{\'et},p,\tilde{x}}$ and hence $s_{\alpha,B,\tilde{x}}$. Therefore, $I_x$ is naturally a subgroup of $G$. The Mumford--Tate group of $\mathcal{A}_{\tilde{x}}$ is a subgroup of $G$ which commutes with $I_x$. By \cite[Lemma 2.1.7]{Ki2}, $I_x$ has the same rank as $G$, hence the Mumford--Tate group of $\mathcal{A}_{\tilde{x}}$ is contained inside a maximal torus of $G$, so that $\tilde{x}$ is a special point.
		
		For the uniqueness, we assume that $x$ is defined over $\F_{p^n}$. Suppose $\tilde{x}'$ is another such $\Ok_L$ point, then $I_x(\Q)$ acts on $\mathcal{A}_{\tilde{x}'}$, and hence $\gamma$ does. Since $\gamma$ maps to $\gamma_p=n\nu(p)\in J_b(\Q_p)$, we have $\gamma_p$ acts $\pdiv_{\tilde{x}'}$, thus by the proof of Proposition \ref{canliftcrystal}, we have $\pdiv_{\tilde{x}'}=\pdiv_{\tilde{x}}$.
	\end{proof}
	
	\begin{definition}Let $x\in \mathscr{S}_{K_p}(G,X)(\Fpbar)$ be a point in the $\mu$-ordinary locus. The abelian variety $\mathcal{A}_{\tilde{x}}$ constructed in the previous theorem is the canonical lift of $\mathcal{A}_x$.
	\end{definition}
	Note that we have now proved the first two assertions of Theorem \ref{main}.  
	
	\subsection*{Shifted subcascades and density of CM lifts}
	The concept of a torsion point makes sense for cascades, and is defined inductively. We recall the definition from \cite[\S 2.2.4]{Mo}.
	For $\Gamma$ an $r$ cascade, we say that $x \in \Gamma_{(1,r)}(S)$ is a torsion point if:
	\begin{enumerate}
		\item $\lambda_{(1,r)}(x)$ and $\rho_{1,r}(x)$ are torsion points in $\Gamma_{(1,r-1)}(S)$ and $\Gamma_{(2,r)}(S)$. 
		\item $x$ is a torsion point of $\Gamma_{(1,r)}$ when viewed as a group over either $\Gamma_{(1,r-1)}$ or $\Gamma_{(2,r)}$.
	\end{enumerate}
	Note that it suffices to check that $\lambda(x)$ is a torsion point in $\Gamma_{(1,r-1)}(S)$, and that $x$ is a torsion point of $\Gamma_{(1,r)}$ when viewed as a group over $\Gamma_{(1,r-1)}$. This follows from the definition of a bi-extension when $r = 3$, and the general case follows by induction on $r$ (indeed, the proof of \cite[Proposition 2.3.12, (ii)]{Mo} uses this claim implicitly).

	\begin{remark}\label{torsionCM}
		When $\Gamma = \Gamma(\pdiv)$, then just as in \cite[Proposition 2.3.12, (ii)]{Mo}, if an $S$ point of $\Gamma_{(1,r)}$ is a torsion point, the corresponding $p$-divisible group is isogenous to $\pdiv^{\can}$. 
	\end{remark}
	We now turn to the final assertion of Theorem \ref{main}. We recall notation from previous sections. Let $x\in \mathscr{S}_{K_p}(G,X)(\Fpbar)$ be a point in the $\mu$-ordinary locus, and let $\pdiv = \pdiv_x$ be the associated $p$-divisible group. Let $\pdiv = \displaystyle{\prod_{i=1}^r \pdiv_i}$ according to its slope decomposition. Recall that each $\pdiv_i$ is isoclinic of slope $\lambda_i$ with  $\lambda_1 > \lambda_2 \hdots > \lambda_r$ and we had defined $\pdiv_{(a,b)} = \displaystyle{\prod_{i = a}^b \pdiv_i}$. The slope decomposition lifts to $\pdiv^{\can}$, the canonical lift of $\pdiv$ and we had defined $\pdiv_{?}^{\can}$ to be the canonical lifts of $\pdiv_{?}$ where $?$ stands for $i$, or $(a,r)$. Let $\Def(\pdiv)$ denote the deformation space over $\Spec \Ok_L$ of $\pdiv$, and let $\Def_G$ denote the space of $G$-adapted deformations -- by Theorem \ref{cmhyp} $\Def_G$ is a shifted subcascade. Recall that $\Defab$ is the subspace of the deformation space of $\pdivab$ described in \S 4.
	
	\begin{thm}\label{torsionpointsdense}
	The space $\Def_G$ contains a dense set of torsion points. Consequently, the final assertion of Theorem \ref{main} is true. 
	\end{thm}
	\begin{proof}
    By remark \ref{torsionCM}, deformations corresponding to torsion points are isogenous to the canonical lift, and are therefore CM points. This reduces the final assertion of Theorem \ref{main} to the density of torsion points in $\Def_G$. 
	We prove the density over the following steps: 
	\begin{enumerate}
	    
	   \item $\Defab$ has a dense set of torsion points if $b = a+1$.
	   
	   \item The fibers of $\Def_{(a-1,b)}$ and $\Def_{(a,b+1)}$ over $0\in \Defab$ have a dense set of torsion points.

	    \item The fibers of $\Def_{(a-1,b)}$ and $\Def_{(a,b+1)}$ over a torsion point of $\Defab$ have a dense set of torsion points.
	\end{enumerate}
	Inductively applying these steps yields the required result. We note that the second step is implied by the third step. However, we first prove the statement in the case when the torsion point is trivial, before turning to the general case. 
	We also note that $\Def_{(a,a+1)}$ and the fibers of $\Def_{(a-1,b)}$ and $\Def_{(a,b+1)}$ over $0 \in \Defab$ are all canonically formal groups over $\Ok_L$. In order to prove that these formal groups have a dense set of torsion points, it suffices to prove that these groups are all $p$-divisible formal groups. This is proved in Propositions \ref{cho} and \ref{tpdivfiber}. The final step is proved in Proposition \ref{finall}.
	
	\end{proof}
	
	\begin{prop}\label{cho}
		The formal group $\Def_{(a,a+1)}$ is $p$-divisible, and hence has a dense set of torsion points.
	\end{prop}
	\begin{proof}
		The proof is an exercise in putting together results proved above, and a theorem of Chai. In order to prove that $\Def_{(a,a+1)}$ is $p$-divisible, it suffices to prove the analogous statement modulo $p$. To that end, let $\Ext_p(\pdiv_{a+1},\pdiv_a)$ be the subfunctor of the mod-$p$ deformation space of $\pdiv_{(a,a+1)}$, consisting of extensions of $\pdiv_{a+1}$ by $\pdiv_a$. According to \cite[\S 2]{Chai},  $\Ext_p(\pdiv_{a+1},\pdiv_a)$ is smooth formal group, having a maximal $p$-divisible formal subgroup $\Ext_p^{\ppdiv}$. Further, by \cite[\S 3.1]{Chai}, we have that $\Ext_p^{\ppdiv}$ equals the central leaf of the equicharacteristic deformation space of $\pdiv_{(a,a+1)}$. 
		
		Proposition \ref{groupsisom} and Wortman's result Proposition \ref{Wort}, yield that $\Spf \ R_G \times_{\Z_p} \F_p$ is contained within the central leaf of $\Def(\pdiv) \times_{\Z_p}\F_p$. Consequently, $\Def_{(a,a+1)} \times_{\Z_p}\F_p$ is contained inside the central leaf of the deformation space of $\pdiv_{(a,a+1)}$, and hence is $p$-divisible. 
		
		The result follows, because every $p$-divisible formal group over $\Ok_L$ has a dense set of torsion points.
	\end{proof}
	
	\begin{prop}\label{tpdivfiber}
		The fibers of $\Def_{(a-1,b)}$ and $\Def_{(a,b+1)}$ over $0 \in \Defab$ are $p$-divisible, and hence have a dense set of torsion points. 
	\end{prop}
	\begin{proof}
		The same proof works for both $\Def_{(a-1,b)}$ and $\Def_{(a,b+1)}$. Therefore, let $H = \Spf R_H$ denote the fiber of $\Def_{(a-1,b)}$ over $0$. Because $\pdivab^{\can} = \displaystyle{\prod_{i=a}^b\pdiv_i^{\can}}$, we have $\Ext(\pdivab^{\can},\pdiv_{a-1}^{\can}) =\prod_{i=a}^b \Ext(\displaystyle{\pdiv_i^{\can}}, \pdiv_{a-1}^{\can})$. It suffices to prove that the image of $H$ in $\Ext(\pdiv^{\can}_i,\pdiv_{a-1}^{\can})$ is $p$-divisible for $a \le i \le b$. Again, it suffices to prove this modulo $p$. Let $H_p$ and $\Ext_p$ denote the mod $p$ reduction of $H$ and $\Ext$.
		
		Notice that the map from $H$ to $\Ext(\pdiv^{\can}_i,\pdiv_1^{\can})$ is induced by pushing out by the projection $\displaystyle{\prod_{j=a}^b\pdiv_j^{\can}} \rightarrow \pdiv^{\can}_i$. By Propositions \ref{groupsisom} and \ref{Wort}, the image of $H_p$ in $\Ext_p(\pdivab^{\can},\pdiv_{a-1}^{\can})$, and therefore in each $\Ext_p(\pdiv_i^{\can},\pdiv_{a-1}^{\can})$, is contained in the central leaf of the extension space. The result follows from \cite[3.1]{Chai}. 
	\end{proof}
	
	Before turning to Proposition \ref{finall}, we will need the following results and discussion. 
	
	\begin{prop}\label{tcosetfiber}
		Let $x \in \Def_{(a,b)}(\Ok_K)$, and suppose that $\pdivh_{(a,b)}$ denote the corresponding $p$-divisible group. Suppose that $F_x$ is the fiber of $\Def{(a-1,b)}$ over $x$. Then $F_x \subset \Ext(\pdivh_{(a,b)},\pdiv_{a-1}^{\can})$ is a coset of some subgroup $H$, which is $p$-divisible. The analogous result holds for the fiber of $\Def_{(a,b+1)}$ over $x$.
	\end{prop}
	\begin{proof}
		That there exists $H$ such that $F_x$ is a coset of $H$ follows because $\Def_G$ is a shifted subcascade. We have to prove that the subgroup is $p$-divisible. Let $y \in F_x$ be some $\Ok_K$ point. We have $F_x = H + y$. We have $y$  reduces modulo $p$ to the unique closed point of $\Def_{(a-1,b)}$. Therefore, $F_x$ and $H$ have the same reduction modulo $p$. It therefore suffices to prove that the reduction modulo $p$ of $F_x$ is a $p$-divisible formal group. 
		
		We have that $x$ and $0$ reduce to the same point modulo $p$ (namely the unique closed point of $\Def_{(a,b)})$. Therefore, the mod $p$ reductions of $F_x$ and $F_{0}$ are the same. The result follows from Proposition \ref{tpdivfiber}. 
	\end{proof}
	
	It suffices to prove that if $x \in \Defab(\Ok_K)$ is a torsion point, then the fiber $F_x \subset \Def{(a-1,b)}$ is the translate of the $p$-divisible group $H$ by a torsion point. 
	
	Recall that the slopes of the $\pdiv_i$ are $\lambda_i$. Let $c \in \mathbb{N}$ be such that $c\lambda_i \in \Z$, and let $\lambda'_i = c\lambda_i$. Let $\pdivh \rightarrow S$ be a $p$-divisible group corresponding to an $S$-point of $\Defab$. We inductively define the following sequence of $p$-divisible groups $\pdivh(i)$ for $i = b-a+1, \hdots, 1$: $\pdivh(b-a+1)$ is defined to be $\pdivh$. Suppose that $\pdivh(b-a+1), \hdots, \pdivh(k)$ have been defined, and each group still has the slope filtration. We now define $\pdivh(k-1)$ as follows: $\pdivh(k)$ has a slope filtration by assumption, and therefore is an extension of $\pdivh(k)_{(k,b)}$ by $\pdivh(k)/\pdivh(k)_{(a,k-1)}$. We multiply this extension by $p^{\lambda'_k - \lambda'_{k-1}}$, and define $\pdivh(k-1)$ to be the resulting $p$-divisible group. It follows that $\pdivh(k-1)$ has a canonical slope filtration induced by $\pdivh(k)$. All these groups are isogenous to each other. 
	
	We now specialise to the case where $\pdivh \rightarrow S$ is $\pdivab^{\uu} \rightarrow \Defab$. Each $\pdiv^{\uu}(n)$ gives a map $\eta_n: \Defab \rightarrow \Def(\pdivab)$. We have the following result:
	
	\begin{prop}
		The map $\eta_1: \Defab \rightarrow \Def(\pdivab)$ factors through the natural inclusion $\Defab\subset \Def(\pdiv_{(a,b)})$. 
	\end{prop}
	\begin{proof}
		It suffices to prove that if $\pdivh$ is a $\Ok_L$-point of $\Defab$, then so is $\pdivh(1)$. Suppose that the filtration on $\pdivh$ is given by $u\mu u^{-1}$, where $u\in \Uab^\circ(\Ok_L)$. By applying Proposition \ref{twoslopecalc} iteratively, we see that the filtration on $\pdivh(1)$ is given by $u'\mu u^{-1}$ where $u'$ is the conjugate of $g$ by $c\nu(p)$. The result follows from the last assertion of Proposition \ref{rest}. 
	\end{proof}
	
	\begin{prop}\label{finall}
		Let $x \in \Defab(\Ok_K)$ be a torsion point. Then the fiber $F_x \subset \Def_{(a-1,b)}$ over $x$ has a dense set of torsion points. The analogous claim holds for $\Def_{(a,b+1)}$.
	\end{prop}
	\begin{proof}
		We will be content with proving the result only for $\Def_{(a-1,b)}$. 
		
		Let $\pdivh$ be the $p$-divisible group corresponding to the point $x$. Let $x(1) \in \Defab$ be the point corresponding to $\pdivh(1)$. Then, $\eta_1: \Def_{(a-1,b)} \rightarrow \Def_{(a-1,b)}$ maps $F_x$ to $F_{x(1)}$. There is a homomorphism $\zeta_1: \Ext(\pdivh,\pdiv_{a-1}^{\can}) \rightarrow \Ext(\pdivh(1),\pdiv_{a-1}^{\can})$ which divides $p^m$, such that when restricted gives $\eta_1: F_x \rightarrow F_{x(1)}$.
		
		Pick the integer $c$, such that $x(1)$ is canonical lift. The corresponding fiber $F$ has a dense set of torsion points by Proposition \ref{tpdivfiber}.
		
		By Proposition \ref{tcosetfiber}, $F_x = H + P$, where $H \subset \Ext(\pdivh,\pdiv_{a-1}^{\can})$ is a formal subgroup which is $p$-divisible, and $P$ some point. So, $\eta_1(H + P) = \zeta_1(H + P) = \zeta_1(H) + \zeta_1(P) \subset F$. As $F$ is a subgroup, we have that $\zeta_1(H) \subset F$. By Proposition \ref{rest}, the dimension of $F_x$ equals the dimension of $F$. As $H$ is $p$-divisible and $\zeta_1$ is a homomorphism dividing $p$, we have that $\zeta_1(H) = F$, and thus $\eta_1(F_x) = F$. 
		
		Because $\eta_1$ takes a point to an isogenous point, we have that $F_x$ has at least one special point, which thus has to be a torsion point. Because $F_x$ is an $H$-coset of $\Ext(\pdivh,\pdiv_{a-1}^{\can})$ with $H$ a $p$-divisible formal group, $F_x$ must have a dense set of torsion points (as that is true for $H$). Therefore, $F_x$ has a dense set of torsion points. 
	\end{proof}

\end{document}